\documentclass{article}
\usepackage[utf8]{inputenc}
\usepackage{xcolor}
\usepackage{hyperref}

\usepackage{float}
\usepackage{amsmath,amsthm,amssymb,amsfonts,hyperref,mathtools}
\usepackage{mathrsfs}
\usepackage{tikz,pgfplots}

\newcommand{\oN}{\mathbb{N}}
\newcommand{\oR}{\mathbb{R}}
\newcommand{\R}{\mathbb{R}}

\newtheorem{theorem}{Theorem}
\newtheorem{lemma}[theorem]{Lemma}
\newtheorem{proposition}[theorem]{Proposition}
\newtheorem{corollary}[theorem]{Corollary}

\newtheorem{example}[theorem]{Example}
\theoremstyle{remark}
\newtheorem{remark}[theorem]{Remark}

\newcommand{\supp}{\text{\rm Supp}}
\newcommand{\wh}{\widehat}
\newcommand{\olE}{{\overline E}}
\newcommand{\M}{{\mathcal M}}
\newcommand{\meas}{{\mathscr M}}
\newcommand{\bg}{\mathbf g}
\newcommand{\N}{\mathbb N}
\newcommand{\MI}{{\mathcal I}}
\newcommand{\wV}{\wh V}
\newcommand{\wG}{\wh G}
\newcommand{\wt}{\widetilde}

\newcommand{\whp}{{\wh p}}
\newcommand{\MS}{{\mathcal S}}
\newcommand{\rank}{\text{\rm rank}}
\newcommand{\cprank}{\rank_{\text{\rm cp}}}
\newcommand{\nnrank}{\rank_+}

\newcommand{\val}{\text{\rm val}}
\newcommand{\valid}{\val}
\newcommand{\valsp}{\val^{\text{\rm isp}}}
\newcommand{\valcp}{\val_{\text{\rm cp}}}

\newcommand{\xiid}{\xi}
\newcommand{\xisp}{\xi^{\text{\rm isp}}}

\newcommand{\wtxisp}{\widetilde{\xi}^{\text{\rm isp}}}

\newcommand{\xicp}{\xi^{\text{\rm cp}}}
\newcommand{\xicpid}{\xi^{\text{\rm cp}}}
\newcommand{\xicpsp}{\xi^{\text{\rm cp,isp}}}
\newcommand{\xicpwsp}{\xi^{\text{\rm cp,wisp}}}

\newcommand{\xinn}{\xi^{+}}
\newcommand{\xinnid}{\xi^{+}}
\newcommand{\xinnsp}{\xi^{+,\text{\rm isp}}}

\newcommand{\taucp}{\tau_{\text{\rm cp}}}
\newcommand{\taunn}{\tau_{+}}
\newcommand{\taucpsos}{\tau^{\text{\rm sos}}_{\text{\rm cp}}}
\newcommand{\taunnsos}{\tau^{\text{\rm sos}}_{+}}

\newcommand{\bc}{{\text{\rm bc}}}
\newcommand{\nzd}{\text{\rm nzd}}
\newcommand{\RG}{\text{\rm RG}}

\newcommand{\ecc}{\text{\rm c}}
\newcommand{\ignore}[1]{}

\usepackage[normalem]{ulem}

\providecommand{\keywords}[1]{\textbf{\textit{Keywords}} #1}

\usepackage{soul}

\title{Exploiting ideal-sparsity in the generalized moment problem  with application to matrix factorization ranks}

\author{Milan Korda
\thanks{LAAS-CNRS and Institute of Mathematics, Toulouse, France. Faculty of Electrical Engineering, Czech Technical University in Prague, Czech Republic.  \url{milan.korda@laas.fr}}
\and
Monique Laurent
\thanks{Centrum Wiskunde \& Informatica (CWI), Amsterdam, and Tilburg University. \url{monique.laurent@cwi.nl} }
\and
Victor Magron
\thanks{LAAS-CNRS and Institute of Mathematics, Toulouse, France. \url{vmagron@laas.fr}}
\and
Andries Steenkamp
\thanks{Centrum Wiskunde \& Informatica (CWI), Amsterdam. \url{andries.steenkamp@cwi.nl}
\newline
This work is supported by the European Union's Framework Programme for Research and Innovation Horizon
2020 under the Marie Skłodowska-Curie Actions Grant Agreement No. 813211  (POEMA), the Czech Science Foundation (GACR) under contract No. 20-11626Y, by the AI Interdisciplinary Institute ANITI funding, through the French “Investing for the Future PIA3” program under the Grant agreement n$^\circ$ ANR-19-PI3A-0004 as well as by the National Research Foundation, Prime Minister’s Office, Singapore under its Campus for Research Excellence and Technological Enterprise (CREATE) programme.
}
}

\begin{document}

\maketitle

\begin{abstract}
We explore a new type of sparsity for the generalized moment problem (GMP) that we call  {\em ideal-sparsity}. 
In this setting, one optimizes over a measure 
restricted to be supported on the variety of an ideal generated by quadratic bilinear  monomials.
We show that this restriction enables an equivalent sparse reformulation of the GMP, where the single (high dimensional) measure variable is replaced by several (lower dimensional) measure variables  supported on the maximal cliques of the graph corresponding to the quadratic bilinear constraints.
We explore the resulting hierarchies of moment-based relaxations for the original dense formulation of GMP and  this new, equivalent ideal-sparse reformulation, when applied to the problem of bounding nonnegative- and completely positive matrix factorization ranks. 
We show that the ideal-sparse hierarchies provide bounds that are at least as good (and often tighter) as those obtained from the dense hierarchy. This is in sharp contrast to the situation when exploiting correlative sparsity, as is most common in the literature, where the resulting bounds are weaker than the dense bounds. Moreover, while correlative sparsity requires the underlying graph to be chordal, no such assumption is needed for ideal-sparsity.
Numerical results show  that the ideal-sparse bounds are often tighter and much faster to compute than their dense analogs.
\end{abstract} 

\noindent

\section{Introduction}\label{sec:intro}

We consider the {\em generalized moment problem} (abbreviated as GMP), of the form
\begin{align}\label{opt:dense}
\valid:=\inf\limits_{\mu\in \meas(\R^n)} \Big\{\int f_0d\mu: \int f_id\mu =a_i\ (i\in [N]),\ \supp(\mu)\subseteq K\Big\},
\end{align}
where $f_0,f_i\in \R[x]$ are multivariate polynomials in the variables $x = (x_1,...,x_n)$, $a_i\in \R$,  $K\subseteq \R^n$ (taken to be Borel measurable), and the optimization is over the set $\meas(\R^n)$ of (finite positive) Borel measures on $\R^n$. In problem (\ref{opt:dense}) we restrict to measures $\mu\in \meas(\R^n)$ whose support $\supp(\mu)$ is contained in $K$, which is equivalent to requiring $\int fd\mu=\int_K fd\mu$ for any Borel measurable function $f:\R^n\to\R$.
Throughout, we assume that $K$ is  a basic closed  semialgebraic set of the form
\begin{align}\label{eqsetK}
K=\{x\in \R^n: g_j(x)\ge 0\ (j\in [m]), \ x_ix_j=0\ (\{i,j\}\in \olE)\},
\end{align}
where $g_j\in \R[x]$ are polynomials, 
$E$  is a given set of pairs of distinct elements of $V=[n] :=\{1,\ldots,n\}$, and $\olE$ is the following set of pairs 
$$\olE=\{\{i,j\}\,:\, i\in V, j\in V, i\ne j, \{i,j\}\not\in E\}.$$ 
Hence, the set $K$  is contained in the variety of the ideal
\begin{align}\label{eqIE}
I_E := \Big\{\sum_{\{i,j\}\in \olE} u_{ij} x_ix_j: u_{ij}\in \R[x]\Big\}\subseteq\R[x]
\end{align}
generated by the monomials $x_ix_j$ for the pairs  $\{i,j\}\in \olE$.
It will be convenient to consider the graph $G=(V,E)$,  so that the conditions $x_ix_j=0$ appearing in the definition of $K$ correspond to the {\em nonedges} of~$G$. This notation may seem at first sight cumbersome. However, the motivation for it  is that the graph $G$  functions as supporting solutions for problem (\ref{opt:dense}); this will be  especially useful for applications to matrix factorization ranks like the completely positive rank (cp-rank) or the nonnegative rank of a matrix, where $G$ will correspond to the support graph of the matrix.

The generalized moment problem (with $K$ semialgebraic) has been much studied in recent years. It permits to model a wide variety of problems, including polynomial optimization (minimization of a polynomial or rational function over $K$), volume computation, control theory, option pricing in finance, and much more. See, e.g.,  \cite{Las2008,Las2009,Las2015,Las2018} and further references therein. 

\medskip
The focus of this paper is to exploit the presence of  explicit ideal constraints (of a special form) in the description of the semialgebraic set $K$ for solving problem (\ref{opt:dense}). This indeed naturally implies some sparsity structure on  problem (\ref{opt:dense}),  to which we will refer as {\em ideal-sparsity} structure. 
Our objective is to explore how one can best exploit this ideal-sparsity structure in order to define more efficient semidefinite hierarchies for problem (\ref{opt:dense}) and apply them to sparse matrix factorization ranks. A remarkable feature is that the ideal-sparse hierarchies provide bounds that are at least as good (and often better) as  the bounds provided by the original dense hierarchy. Moreover, the underlying sparsity graph is not required to be chordal. Both these features are in stark contrast to the existing sparse hierarchies based on correlative sparsity whose bounds are always dominated by the dense bounds and that require the underlying sparsity graph to be chordal in order to guarantee convergence. We refer to Section \ref{seccsp} for an in-depth discussion about correlative and ideal-sparsity.

We focus here on the application to the completely positive and the nonnegative factorization ranks, asking for a factorization by nonnegative vectors. However, as we will mention in the final discussion section,   this ideal-sparsity framework  could also be applied to more general settings. Indeed, it could be applied to other matrix factorization ranks, such as the (completely) positive semidefinite rank, where  one asks for a factorization by positive semidefinite matrices, in which case one would have to apply tools from polynomial optimization in noncommutative variables. Also, instead of an ideal generated by quadratic monomials, one could have an ideal generated by higher degree monomials. In addition, up to a change of variables, one could consider an ideal generated by more general products of linear terms, such as $(a^Tx+b)(c^Tx+d)$. This type of constraint, often known as a {\em complementary constraint}, occurs in various applications, including ReLU neural networks or optimization when considering KKT optimality conditions.
 
 \medskip
Next, we mention the overall  organization of the paper and give some general notation used throughout. After that, we will give a broad overview of the contents and main results obtained in the paper.

\subsection*{Organization of the paper}
The paper is organized as follows.  In the rest of the Introduction we outline the main results in the paper. Then, in  Section \ref{sec:prel} we recall some preliminaries about linear functionals on polynomials and moment matrices. 
In Section \ref{sec:sparseGMP} we consider the GMP (\ref{opt:dense}): we show its  sparse reformulation (\ref{opt:sparse}), we present the corresponding sparse hierarchies, and we discuss how ideal-sparsity relates to the more classic correlative sparsity.
Section~\ref{sec:app-cp} is devoted to the application to the cp-rank and Section \ref{sec:app-nn} to the application to the nonnegative rank.
We conclude with some final remarks and discussions in Section~\ref{sec:final}.

\subsection*{Notation}
We gather here some notation that is used throughout the paper.
For $n,t\in \N$ set $\N^n_t=\{\alpha\in \N^n: |\alpha|\le t\}$, where $|\alpha|=\sum_{i=1}^n\alpha_i$ denotes the degree of the monomial $x^\alpha =x_1^{\alpha_1}\cdots x_n^{\alpha_n}$.
We let $[x]_t=(x^\alpha)_{\alpha\in \oN^n_t}$ denote the vector of monomials with degree at most $t$ (listed in some given fixed order). Moreover, $\R[x]$ (resp., $\R[x]_t$) denotes the set of $n$-variate polynomials in variables $x=(x_1,\ldots,x_n)$ (with degree at most $t$).
Let $\Sigma$ denote the set of sum-of-squares polynomials, of the form $\sum_iq_i^2$ for some $q_i\in \R[x]$, and set $\Sigma_t=\Sigma\cap \R[x]_t$.

Consider a set $U\subseteq [n]$. Given a vector $y\in \R^{|U|}$, we let $(y,0_{V\setminus U})\in \R^{n}$ denote the vector obtained by padding $y$ with zeros at the entries indexed by $[n]\setminus U$.
For an $n$-variate function $f:\R^{|V|}\to\R$, we let $f_{|U}:\R^{|U|}\to\R$ denote the function in  the variables $x(U)=\{x_i: i\in U\}$, which is obtained from $f$ by setting to zero all the variables $x_i$ indexed by $i\in V\setminus U$. That is, $f_{|U}(y)=f(y,0_{V\setminus U})$ for $y\in\R^{|U|}$. 
So, if $f$ is an $n$-variate polynomial, then $f_{|U}$ is a $|U|$-variate polynomial in the variables $x(U)$.

For a symmetric matrix $M\in \mathcal S^n$, the notation $M\succeq 0$ means that $M$ is positive semidefinite, i.e., $v^TMv\ge 0$ for all $v\in\R^n$. Throughout, we let $I_n$ and $J_n$ denote the identity matrix and the all-ones matrix of size $n$, which we sometimes also denote as $I$ and $J$ when the dimension is clear from the context.
The support of a vector $x\in\R^n$ is the set $\supp(x)=\{i\in [n]: x_i\ne 0\}$.

 \subsection*{Roadmap through the paper}
\medskip In the rest of this section, we now offer a quick roadmap through the main contents of the paper. We begin with recalling how to define the dense hierarchy of bounds for the problem (\ref{opt:dense}). We then  discuss their main drawback (quick growth of the matrices involved in the semidefinite programs) and several options for addressing this difficulty that have been offered in the literature. After that, we introduce the new ideal-sparse reformulation of problem (\ref{opt:dense}) and the corresponding ideal-sparse hierarchy, which we then specialize to the applications for bounding the completely positive and nonnegative ranks. 
 
\subsection*{Classical (dense) moment relaxations} 

We begin by recalling the classical moment approach that permits to build hierarchies of semidefinite approximations for problem (\ref{opt:dense}). For details, see, e.g., the monograph by Lasserre \cite{Las2009}, or the survey \cite{dKL19}.
 For  $t\in\N\cup\{\infty\}$, the set 
\begin{equation}\label{eqQMt}
\M(\bg)_{2t}=\Big\{\sum_{j=0}^m\sigma_jg_j: \sigma_j\in \Sigma,\ \deg(\sigma_jg_j)\le 2t\Big\}\subseteq \R[x]_{2t}
\end{equation}
 is  the quadratic module generated by $\bg=(g_1,\ldots,g_m)$, and truncated at degree $2t$, setting $g_0=1$. We also set $\M(\bg)=\M(\bg)_\infty$. Here,  $\Sigma$ denotes the set of sums of squares of polynomials in $\R[x]$. 
Similarly, 
\begin{equation}\label{eqIt}
I_{E,2t}=\Big\{\sum_{\{i,j\}\in \olE} u_{ij} x_ix_j: u_{ij}\in \R[x]_{2t-2}\Big\}\subseteq\R[x]_{2t}
\end{equation}
denotes the truncation of the ideal $I_E$ at degree $2t$.
We can now define the moment relaxation of level $t$ for problem (\ref{opt:dense}):
\begin{equation}\label{eqxiidt}
\begin{array}{ll}
\xiid_t:= \inf\{L(f_0): & L\in \R[x]^*_{2t}, \\ & L(f_i)= a_i \ (i\in [N]),\\
& L\ge 0 \text{ on } \M(\bg)_{2t},\ L=0 \text{ on } I_{E,2t}\}.
\end{array}
\end{equation} 
Here, $\R[x]^*_{2t}$ denotes the set of linear functionals $L:\R[x]_{2t}\to \R$. The motivation for the above parameter is as follows. Assume $\mu\in \M(\R^n)$ is a measure that is feasible for problem (\ref{opt:dense}), and 
consider the associated linear functional $L$ that acts on $\R[x]_{2t}$ via integration: $p\in \R[x]_{2t}\mapsto L(p)=\int pd\mu$. Then, it is easy to see that $L$ is feasible for (\ref{eqxiidt}):
$L(f_i)=\int f_id\mu=a_i$, $L\ge 0$ on $ \M(\bg)_{2t}$ (since any polynomial in $\M(\bg)_{2t}$ is nonnegative on the set $K$), and 
$L=0$ on $I_{E,2t}$ (since any polynomial in $ I_{E,2t}$ vanishes on $K$). This shows that the parameter $\xiid_t$ lower bounds the optimum value $\valid$ of problem (\ref{opt:dense}). 

We refer to the above hierarchy of parameters $\xiid_t$ as the {\em dense moment hierarchy}. 
Clearly, they satisfy  $\xiid_t\le \xiid_{t+1}\le \xiid_\infty\le \valid$. 
Moreover,  under some mild assumptions, these bounds converge asymptotically to the optimum value $\valid$ of (\ref{opt:dense}).
This fundamental property follows from the general theory about GMP (see, e.g., 
\cite{Las2009}, \cite{dKL19}) and is summarized in the following theorem that will be used repeatedly throughout this work.

\begin{theorem}\label{theoconvGMP}
Assume problem (\ref{opt:dense}) is feasible and  the following Slater-type condition holds:
\begin{align}\label{eqCQ}
\text{there exist scalars } z_0,z_1,\ldots,z_N  \in\R \text{ such that } \sum_{i=0}^{N}z_i f_i(x)>0 \text{ for all } x\in K.
\end{align}
Then, problem (\ref{opt:dense}) has an optimal solution $\mu$, which can be chosen to be finite atomic. 
If, in addition, $\M(\bg)$ is Archimedean, i.e., $R-\sum_{i=1}^n x_i^2\in \M(\bg)$ for some scalar $R>0$, then  we have $\lim_{{t}\to \infty}\xiid_t= \xiid_{\infty}=\valid$.
\end{theorem}

As it will be recalled in Section \ref{sec:prelmom}, program (\ref{eqxiidt})  can be reformulated as a semidefinite program and thus the bound $\xiid_t$ can be computed using semidefinite optimization algorithms. However, a common drawback  of the dense hierarchy (\ref{eqxiidt})  is that it involves matrices whose size grows very quickly with the level $t$ and with the degree and number of variables of the polynomials $f_0$, $f_1,\ldots,f_N$, $g_1,\dots,g_m$. 
Hence, even though these relaxations are convex, they might be challenging to solve already for GMP instances of modest size. 

\subsection*{Existing schemes to improve scalability of the dense moment relaxations}

Several schemes have been developed to overcome the scalability issue of the dense hierarchy (\ref{eqxiidt}) just mentioned above.
They aim to reduce the size of the involved matrices by exploiting the specific structure of the input polynomials without compromising the convergence guarantees of the structure-induced moment relaxations. 
One workaround consists of exploiting the symmetries \cite{riener2013exploiting}, but this requires that each input polynomial is invariant under the action of a subgroup of the general linear group.  

Another approach is to exploit different kinds of sparsity structures.  
The first kind is called {\em correlative sparsity}, which occurs when there are few correlations between the variables of the input polynomials \cite{waki2006sums,lasserre2006convergent}. 
Correlative sparsity has been extended to derive moment relaxations of polynomial problems in complex variables \cite{josz2018lasserre},  noncommutative variables \cite{klep2021sparse} and polynomial matrix inequalities \cite{zheng2021sum}. 
The second kind is called {\em term sparsity}, which occurs when they are few (by comparison with all possible) monomial terms involved in the input polynomials, and for which correlative sparsity is not exploitable. 
For unconstrained polynomial optimization, one well-known solution is to eliminate the monomial terms which never appear among the support of sums of squares decompositions \cite{reznick1978extremal}.
Alternatively, one can decompose the input polynomial as a sum of nonnegative circuits, by solving a geometric programming relaxation \cite{iliman2016amoebas} or a second-order cone programming relaxation  \cite{averkov2019optimal,wang2020second}, or as a  sum of arithmetic-geometric-mean-exponentials  \cite{chandrasekaran2016relative} with relative entropy programming relaxations.
Term sparsity has recently been the focus of active research with extensions to constrained polynomial optimization  \cite{wang2021chordal,wang2021tssos}. 
Note that both kinds of sparsity can be combined \cite{cstssos}. For a general exposition about sparse polynomial optimization, we refer to the recent surveys \cite{sparsebook,zheng2021review}.

We will return to the correlative sparsity approach for GMP later in Section \ref{seccsp} and discuss  how it relates to the new ideal-sparsity structure considered in the paper.
By contrast with classical polynomial optimization problems, it is not
completely clear which initial set of monomials should be chosen to
initialize the term sparsity hierarchy when facing a given GMP
instance. Therefore, we do not explore the combination of term and
ideal-sparsity, for such an investigation would warrant a separate
publication.

\subsection*{New ideal-sparse moment relaxations}\label{sec:intro-sparse}

As we now see, one can  exploit the fact that the set $K$ in (\ref{eqsetK}) is contained in the variety of the ideal $I_E$ from (\ref{eqIE}). The basic idea is that, instead of optimizing over a {\em single} measure $\mu$ supported on  $K\subseteq \R^n$, one may  optimize over {\em several} measures that are supported on smaller dimensional spaces.

A set $W \subseteq V$ is a clique of the graph $G=(V,E)$ if $\{u,v\}\in E$ for any two distinct vertices  $u,v\in W$.
A clique is maximal (w.r.t inclusion) if it is not strictly contained in any other clique of $G$.
Let $V_1,\ldots,V_p$ denote the maximal cliques of the graph $G=(V,E)$ and,
for $k\in [p]$, define  the following subset of $K$:
\begin{equation}\label{eqwhKk}
\wh {K_k} :=\{x\in K: \supp(x)\subseteq V_k\}\subseteq K\subseteq \R^n.
\end{equation}
Recall $\supp(x)=\{i\in [n]: x_i\ne 0\}$ denotes the {support} of $x\in \R^n$.
If $x\in K$, then its support $\supp(x)$ is a clique of the graph $G$ and thus it is contained in a maximal clique $V_k$,  so that  $x\in \wh{K_k}$ for some $k\in [p]$.
Therefore, the sets $\wh{K_1},\ldots,\wh{K_p}$ cover the set $K$:
\begin{equation}\label{eqKkcover}
K=\wh{K_1}\cup \ldots \cup \wh{K_p}. 
\end{equation}

\noindent
  Next, define the projection $K_k\subseteq \R^{|V_k|}$ of $\wh{K_k}$ onto the subspace indexed by $V_k$:
 \begin{align}\label{eqKk}
K_k:=\{y\in \R^{|V_k|}: (y,0_{V\setminus V_k})\in \wh{K_k}\} \subseteq \R^{|V_k|}.
\end{align}
Recall that $(y,0_{V\setminus V_k})$  denotes the vector of $\R^n$ obtained from $y\in \R^{|V_k|}$ by padding it with zeros at all entries indexed by $V\setminus V_k$.
Moreover, given a function $f: \oR^{|V|}\to \oR$, the function $f_{|V_k}(y): \oR^{|V_k|}\to \oR$ is defined by 
$f_{|V_k}(y)=f(y,0_{V\setminus V_k})$ for $y\in\R^{|V_k|}$.
We may now define the following sparse analog of problem~(\ref{opt:dense}):
\begin{align}\label{opt:sparse}
\valsp:= \inf_{\mu_k\in \meas(\R^{|V_k|}), k\in [p]} \Big\{\sum_{k=1}^p \int  {f_0}_{|V_k}d\mu_k: \sum_{k=1}^p \int {f_i}_{|V_k} d\mu_k=a_i \ (i\in [N]), \ \supp(\mu_k)\subseteq K_k\ (k\in [p])\Big\}.
\end{align}
Hence, while problem (\ref{opt:dense}) has a single measure variable $\mu$ on the space $\R^{|V|}$,  problem (\ref{opt:sparse}) involves $p$ measure variables, where $\mu_k$ is on the smaller dimensional space $\R^{|V_k|}$. As we will show in Proposition \ref{propequiv} below, both formulations (\ref{opt:dense}) and (\ref{opt:sparse}) are in fact equivalent, i.e, we have equality $\valid=\valsp$.
Here, we use the superscript `isp' as a reminder that the formulation exploits ideal-sparsity; we will follow this same notation below for the corresponding moment hierarchy and also later for the parameters attached to matrix factorization ranks.

Based on its reformulation via (\ref{opt:sparse}), we can now define another hierarchy of moment approximations for problem (\ref{opt:dense}), to which we refer as the {\em  ideal-sparse moment hierarchy}:
\begin{equation}\label{eqxispt}
\begin{array}{ll}
\xisp_t:=\inf\Big \{\sum_{k=1}^p L_k({f_0}_{|V_k}): & L_k \in \R[x(V_k)]_{2t}^* \ (k\in [p]),\\
&  \sum_{k=1}^p L_k({f_i}_{|V_k})=a_i\ (i\in [N]),\\
& L_k\ge 0\text{ on } \M(\bg_{|V_k})_{2t} \ (k\in [p])\Big\}.
\end{array}
\end{equation}
This hierarchy provides bounds for $\valid$ that are at least at good as the bounds (\ref{eqxiidt}).
Namely, $$\xiid_t\le \xisp_t \le \valid$$ holds for any $t\ge 1$ (see Theorem \ref{lem:dens<spar} below).

Hence, the ideal-sparse bounds $\xisp_t$ present a double advantage compared to the dense bounds $\xiid_t$. First, they are at least as good and sometimes strictly  better, as we will see later in concrete examples. For the application to the completely positive and nonnegative ranks, we will see classes of matrices showing a large separation between the dense bound and the ideal-sparse bound of level $t=1$; see Examples \ref{excp} and \ref{ex:xi_1_sep}.
Second, their computation is potentially faster since  the sets $V_k$ can be much smaller than the full set $V$. 
We will also see in later examples that the computation of the ideal-sparse bounds can be much faster indeed.  On the other hand, the number of cliques in the graph $G$ could be large, so there is a trade-off. We refer to discussions later in the paper around specific applications.

Interestingly, no structural chordality property needs to be  assumed on the cliques $V_1,\ldots,V_p$ of the graph $G$.
We will comment  in Section \ref{seccsp} about the link between the ideal-sparsity approach presented here and the more classical  correlative sparsity approach  that can be followed when considering a chordal extension $\wh G$ of the graph $G$.

The idea of optimizing over multiple measures has appeared already in several contexts, similarly to what can be routinely done in most computational methods, e.g., finite elements.
In the context of analyzing dynamical systems involving polynomial data, a similar trick has been used to perform optimal control of piecewise-affine systems in \cite{abdalmoaty2013measures}, then later on to characterize invariant measures for piecewise-polynomial systems (see \cite[\S~3.5]{invsdp}).
In the context of set estimation, one can also rely on a multi-measure approach to approximate the moments of Lebesgue measures supported on unions of basic semialgebraic sets  \cite{lasserre2019semidefinite}. 
The common idea consists in using the piecewise structure of the dynamics and/or 
the state-space partition to decompose the measure of interest into a sum of local  measures supported on each partition cell.
The advantage in our current setting is that these measures are supported on smaller dimensional spaces, which leads to potentially strong computational benefit when considering the associated semidefinite programming relaxations.

\medskip
We next present   instances of GMP to which the above ideal-sparsity framework naturally applies, namely to derive bounds on matrix factorization ranks such as the completely positive rank and the nonnegative rank.

\subsection*{Bounds on the completely positive rank via GMP}

 Let $A\in\mathcal S^n$ be a symmetric matrix with nonnegative entries.
Assume $A$ is a completely positive matrix (abbreviated as  {\em cp-matrix}), i.e., 
$A$ can be written as 
\begin{align*}\label{eqA}
A=\sum_{\ell=1}^ra_\ell a_\ell^T\ \text{ for some nonnegative vectors } a_1,\ldots, a_r\in \R^n_+.
\end{align*}
 Then, the smallest integer $r\in\N$ for which such a decomposition exists is  the {\em cp-rank} of $A$, denoted $\cprank(A)$.  Checking whether a given matrix $A$ is completely positive is a computational hard problem (see \cite{DickinsonGijben}). The moment approach has been applied to the question of testing whether $A$ is a a cp-matrix and finding a cp-factorization, in particular, by Nie \cite{Nie14} who formulates it as testing the existence of a representing measure (over the standard simplex) for the sequence of entries of $A$. 

Hierarchies of moment-based relaxations have also been employed to obtain sequences of bounds for the rank of tensors \cite{tang2015guaranteed},
as well as for the symmetric nuclear norm of tensors \cite{nie2017symmetric}.  
Here, we focus on the question of bounding the cp-rank.
No efficient algorithms are known for  finding the cp-rank.  This motivates the search for efficient methods giving lower bounds on the cp-rank, as, e.g.,  in \cite{FP16,GdLL2019a,GLS2022}. 
 The following parameter was introduced in \cite{FP16}, as  a natural ``convexification" of the cp-rank:
 \begin{equation}\label{eqtaucp}
 \taucp (A)=\inf\Big\{\lambda: {1\over \lambda }A \in \text{conv}\{xx^T: x\in \R^n_+,\ A-xx^T\succeq 0, A\ge xx^T\}\Big\},
\end{equation}
providing a lower bound for it: $\taucp(A)\le \cprank(A)$.
As observed below, the parameter $\taucp(A)$ can be reformulated as an instance of problem (\ref{opt:dense}), with an ideal-sparsity structure inherited from the matrix $A$.
 
To avoid trivialities we assume $A_{ii}>0$ for all $i\in [n]$. (Indeed, if $A$ is a  cp-matrix with $A_{ii}=0$, then its $i$-th row/column is identically zero and thus it can be removed without changing the cp-rank.)
Note that the constraints $A\ge xx^T$ and $x\ge 0$ are equivalent\footnote{The reason for using the constraint $\sqrt{A_{ii}}x_i-x_i^2\ge 0$ instead of $A_{ii}-x_i^2\ge 0$ is because it leads to a larger quadratic module and thus a possibly better bound (see \cite{GdLL2019a}).}  to $\sqrt{A_{ii}}x_i-x_i^2\ge 0$ ($i\in [n]$) and $ A_{ij}-x_ix_j\ge 0$ ($1\le i<j\le n$). Moreover, they imply $x_ix_j=0$ whenever $A_{ij}=0$. Let us define the graph $G_A=(V,E_A)$ as  the {\em support graph} of $A$, with  
\begin{equation}\label{eqsetEA}
E_A=\{\{i,j\}: A_{ij}\ne 0,\, i,j\in V,\, i\ne j\},\ \olE_A=\{\{i,j\}: A_{ij}=0,\, i,j\in V,\, i\ne j\},
\end{equation}
 and define the semialgebraic set 
\begin{equation}\label{eqsetKA}
\begin{array}{lll}
K_A=\{x\in \R^n: &\  \sqrt{A_{ii}}x_i-x_i^2\ge 0 \ (i\in [n]), & \ A_{ij}-x_ix_j\ge 0 \ (\{i,j\}\in E_A),\\
 & \ x_ix_j=0 \ (\{i,j\}\in \olE_A),
&\  A-xx^T\succeq 0
\}.
\end{array}
\end{equation}

As we now observe, the parameter $\taucp(A)$ can be reformulated as an instance of GMP.
 
 \begin{lemma} \label{lemtaucpGMP}
 The parameter $\taucp(A)$ is equal to the optimal value of the generalized moment problem:
$$\valcp:=\inf_{\mu\in \meas(\R^n)}\Big \{\int 1d\mu: \int x_ix_j d\mu=A_{ij} \ (i,j\in V), \ \supp(\mu)\subseteq K_A\Big\}.$$
 \end{lemma}
 \begin{proof}
 The (easy) key observation is that any feasible solution to $\taucp(A)$, i.e., any decomposition
of the form $A=\lambda \sum_{\ell=1}^s \lambda_\ell a_{\ell}a_{\ell}^T$ with $\lambda_\ell\ge 0,$ $\sum_{\ell  =1}^s\lambda_\ell=1$ and $a_\ell\in K_A$, corresponds to a measure $\mu:=\lambda\sum_{\ell=1}^s\lambda_\ell \delta_{a_{\ell}}$ that is feasible for $\taucp(A)$ and finite atomic (and vice-versa).
Observe also that the Slater-type condition (\ref{eqCQ}) holds (since 
$f_0=1>0$ on $K_A$). The result now follows using (the first part of) 
 Theorem~\ref{theoconvGMP}: if $A$ is completely positive, then $\valcp$ is feasible and thus has a finite atomic optimal solution, which implies $\taucp(A)=\valcp$; otherwise, both parameters $\taucp(A)$ and $\valcp$ are infeasible and thus equal to $\infty$.
 \end{proof}
 
Based on the formulation of the parameter $\taucp(A)$ in Lemma \ref{lemtaucpGMP} as a GMP instance,  we can define  the corresponding bounds $\xicpid_t(A)$,
 obtained as special instance of the bounds  (\ref{eqxiidt}) (see relations (\ref{eqcpA})-(\ref{eqcpmat}) below).
Then, the convergence of the bounds $\xicpid_t(A)$ to $\taucp(A)$  follows   as a direct application of Theorem \ref{theoconvGMP}.

\medskip
As in the general case of GMP, one may exploit the presence of the ideal constraints $x_ix_j=0$ (for $\{i,j\}\in \olE_A$) in the definition of $K_A$ and define a hierarchy of ideal-sparse bounds 
$\xicpsp_t(A)$. These bounds satisfy 
$$
 \xicpid_t(A)\le \xicpsp_t(A)\le \taucp(A)\ \text{ for any } t\ge 1,$$ also with asymptotic convergence 
 to $\taucp(A)$.
We refer to Section \ref{sec:app-cp} for details about these parameters and links to earlier bounds in the literature.

\subsection*{Bounds on the nonnegative rank via GMP}

The above approach  for the cp-rank  naturally extends to the asymmetric setting of the nonnegative rank.
For a nonnegative matrix $M\in \R^{m\times n}$, its {\em nonnegative rank}, denoted $\nnrank(M)$, is defined as the smallest integer $r$ for which there exist nonnegative vectors $a_\ell \in\R^m_+$ and $b_\ell\in \R^n_+$ such that 
\begin{equation}\label{eqnnM0}
M=\sum_{\ell=1}^r a_\ell b_\ell^T.
\end{equation}
In other words, $\nnrank(M)$ can be seen as  the smallest cp-rank of a cp-matrix $A\in \mathcal S^{m+n}$ of the form
$$A= \left(\begin{matrix} X & M \cr M^T & Y\end{matrix}\right) \text{ for some nonnegative symmetric matrices } X\in \mathcal S^m, Y\in \mathcal S^n.$$
Computing the nonnegative rank is an NP-hard problem \cite{Vav09}. 
In analogy to the parameter $\taucp$ in (\ref{eqtaucp}),  the following ``convexification" of the nonnegative rank  was introduced  in \cite{FP16}:
\begin{equation}\label{eqtaunn}
\tau_+(M)=\inf\Big\{\lambda: {1\over \lambda} M \in \text{conv}\{ xy^T: x\in \R^m_+,\ y\in \R^n_+,\ M \ge xy^T\}\Big\}.
\end{equation}
Note that, compared to the parameter $\taucp(A)$ in (\ref{eqtaucp}), where we had an additional constraint $A-xx^T\succeq 0$,  we now cannot impose such a constraint. 

One can define  analogs of the bounds $\xicpid_t$ and $\xicpsp_t$ for the nonnegative rank, which now involve a linear functional acting on polynomials in $m+n$ variables. For convenience, we set $V=[m+n]=U\cup W$, where $U=[m]=\{1,\ldots,m\}$ (corresponding to the row indices of $M$) and $W=\{m+1,\ldots, m+n\}$ (corresponding to the column indices of $M$, up to a shift by $m$).
We also set 
\begin{equation}\label{eqsetEM}
E^M=\{\{i,j\}\in U\times W:  M_{i,j-m}\ne 0\},\ \olE^M=(U\times W)\setminus E^M=\{\{i,j\}\in U\times W: M_{i,j-m}=  0\},
\end{equation}
so that $E^M$ corresponds to the (bipartite) support graph of the matrix $M$. Note that, in comparison to (\ref{eqsetEA}), we now only consider {\em bipartite} pairs $\{i,j\}$ (with $i\in U$ and $j\in W$). To emphasize the difference between the two situations we now put $M$ as a superscript, while we placed $A$ as subscript in the notation $E_A$. 

Let $M_{\max}=\max_{i,j}M_{ij}$ denote the largest entry of $M$. 
As observed in \cite{GdLL2019a}, one may assume without loss of generality that the vectors in (\ref{eqnnM0}) satisfy 
$\|a_\ell\|_\infty, \|b_\ell\|_\infty \le \sqrt{M_{\max}}$ (after rescaling). 
This motivates defining the following semialgebraic set 
\begin{equation}\label{eqsetKM}
\begin{array}{lll}
K^M=\{x\in \R^{m+n}: & \sqrt{M_{\max}}x_i-x_i^2\ge 0\ (i\in [m+n]), & M_{i,j-m}-x_ix_j \ge 0 \ ( \{i,j\}\in E^M),\\
& x_ix_j=0 \ ( \{i,j\}\in \olE_M\} . &
\end{array}
\end{equation}
The analog of Lemma \ref{lemtaucpGMP} holds, which provides a GMP reformulation for $\taunn(M)$.
\begin{lemma} \label{lemtaunnGMP}
The parameter $\taunn(M)$ is equal to the optimal value of the generalized moment problem:
$$\inf_{\mu\in \meas(\R^{m+n})}\Big \{\int 1d\mu: \int x_ix_j d\mu=M_{i,j-m} \ (i\in U, j\in W), \ \supp(\mu)\subseteq K^M\Big\}.$$
 \end{lemma}

Based on this formulation of the parameter $\taunn(M)$,  we may consider the corresponding bounds $\xinnid_t(A)$, as special instance of the bounds in (\ref{eqxiidt}).
Their asymptotic convergence  to $\taunn(A)$ follows   as a direct application of Theorem \ref{theoconvGMP}.
One may also exploit the presence of the ideal constraints $x_ix_j=0$ (for $\{i,j\}\in \olE^M$) in the definition of $K^M$ and define a hierarchy of sparse bounds 
$\xinnsp_t(M)$. These parameters satisfy 
$$
\xinnid_t(M)\le \xinnsp_t(M)\le \taunn(M)\ \text{ for any } t\ge 1,
$$ 
with asymptotic convergence of all parameters to $\tau_+(M)$.
We refer to Section \ref{sec:app-nn} for details about these parameters.

\section{Preliminaries about sums of squares and moments}\label{sec:prel}

In this section, we recall some preliminaries about sums of squares and linear functionals on polynomials that we will use throughout.
These results are well-known in the polynomial optimization community, we refer, e.g., to the following sources \cite{dKL19,Las2001,Las2008,Las2009,Las2015,Las2018,Laurent2009} and further refereces therein for background and broad overviews.

\subsection{Nonnegative linear functionals and  moment matrices}\label{sec:prelmom}

The program (\ref{eqxiidt}) defining the parameter $\xiid_t$  involves a  linear functional $L\in \R[x]^*_{2t}$, which is assumed to be nonnegative on the truncated quadratic module $\M(\bg)_{2t}$ (in (\ref{eqQMt})) and to vanish on the truncated ideal $I_{E,2t}$ (in (\ref{eqIt})).
We now recall how these conditions can be expressed more concretely in terms of positive semidefiniteness conditions on associated (moment) matrices and thus used to reformulate the program (\ref{eqxiidt}) as a semidefinite program.

For this, given $L\in \R[x]_{2t}^*$, define the matrix 
$$M_t(L):=(L(x^\alpha x^\beta))_{\alpha,\beta\in \N^n_t}= L([x]_t [x]_t^T),$$ often called a {\em (pseudo)moment matrix} in the literature. So, in  the notation $L([x]_t [x]_t^T)$, it is understood that $L$ is acting entry-wise on the entries of the polynomial matrix  $[x]_t [x]_t^T=(x^{\alpha+\beta})_{\alpha,\beta\in\N^n_t}$. Then, it is well-known (and easy to see) that 
$L(\sigma)\ge 0$ for all $\sigma\in\Sigma\cap\R[x]_{2t}$ if and only if the matrix $M_t(L)$ is positive semidefinite.
Consider now a polynomial $g$ with degree $k=\deg(g)$. Then $L(\sigma g)\ge 0$ for all $\sigma\in\Sigma$ with $\deg(\sigma g)\le 2t$ if and only if the matrix $M_{t-\lceil k/2\rceil}(gL):=L(g[x]_{t-\lceil k/2\rceil} [x]_{t-\lceil k/2\rceil}^T)$ (often called a {\em localizing moment matrix}) is positive semidefinite. Hence, the condition $L\ge 0$ on $\M(\bg)_{2t}$ can be equivalently reformulated via the positive semidefiniteness constraints
$$L([x]_t [x]_t^T)\succeq 0,\quad L(g_j [x]_{t-\lceil \deg(g_j)/2\rceil} [x]_{t-\lceil \deg(g_j)/2\rceil}^T) \succeq 0 \ \text{ for } j\in [m].
$$
In the same way, the ideal condition $L=0$ on $I_{E,2t}$ is equivalent to the linear constraints
$$L(x_ix_jx^\alpha)=0\ \text{ for all } \{i,j\}\in \olE \text{ and } \alpha\in \N^n_{2t-2}.
$$
Hence, the parameter $\xiid_t$ is expressed as the optimum value of a semidefinite program. 
Recall that there exist efficient algorithms for solving semidefinite programs up to any precision (under some mild assumptions; see, e.g., \cite{dKV16} and further references therein).

\subsection{Flatness and extraction of optimal solutions} \label{sec:flatty}

As recalled in Theorem \ref{theoconvGMP}, if the quadratic module $\M(\bg)$ is Archimedean (i.e., $R-\sum_i x_i^2\in\M(\bg)$ for some $R>0$), then the bounds $\xiid_t$ converge asymptotically to $\xiid_\infty$. In addition, if the Slater-type condition (\ref{eqCQ}) holds, then $\xiid_\infty=\valid$ and problem (\ref{opt:dense}) has a finite atomic optimal solution $\mu$, i.e., supported on finitely many points in  $K$.

A remarkable property of the bounds $\xiid_t$ is that they often exhibit {\em finite} convergence.
Indeed, there is an (easy to check) criterion, known as the {\em flatness condition}, which permits to conclude that the level $t$ bound is exact, i.e., $\xiid_t=\valid$, and to extract a finite atomic optimal solution of GMP. This flatness condition, see  (\ref{eqflat}) below,  goes back to work of Curto and Fialkow \cite{CF96,CF00}. We also refer, e.g., to \cite{Las2009,Laurent2009} for a detailed  exposition of the following result.
For details  on how to extract an atomic optimal solution under the flatness condition (\ref{eqflat}), we refer to \cite{HL05,Laurent2009}.

\begin{theorem}  \cite{CF96,CF00} \label{CurtoFialkowr_flat_ext_atom}
Consider the set $K$ from (\ref{eqsetK})  and set $d_K=\max\{1, \lceil\deg(g_j)/2\rceil: j\in [m]\}$.
Let $t\in \N$ such that  $2t\ge \max\{\deg(f_i): 0\le i\le N\}$ and  $t\ge d_K$.
Assume  $L\in\R[x]_{2t}^*$ is an optimal solution to the program (\ref{eqxiidt}) defining the parameter $\xiid_t$ and it satisfies the following {\em flatness condition}:
\begin{equation}\label{eqflat}
\rank\ L([x]_s[x]_s^T) =\rank\ L([x]_{s-d_K}[x]_{s-d_K}^T)=:r \ \text{ for some integer } s  \text{ such that } d_K\le s\le t.
\end{equation}
Then, equality $\xiid_t=\valid$ holds, and problem (\ref{opt:dense}) has an optimal solution $\mu$ that is finite atomic and supported on $r$ points in $K$.
\end{theorem}

The above result naturally applies also to the sparse reformulation (\ref{opt:sparse}) of GMP and to the sparse hierarchy $\xisp_t$ in (\ref{eqxispt}).
Indeed, it suffices to apply Theorem \ref{CurtoFialkowr_flat_ext_atom} to each of the linear functionals $L_k$ and to check whether $L_k$ satisfies the corresponding flatness criterion.
We adapt the result to this setting for concreteness. 

\begin{corollary}\label{corsparse}
Consider the sets $K$ in (\ref{eqsetK}) and $K_k$ in (\ref{eqKk}) and set 
$d_{K_k}=\max\{1, \lceil \deg((g_j)_{|V_k})/2\rceil: j\in [m]\}$ for $k\in [p]$.
Let $t\in\N$ such that $2t\ge \max\{\deg(f_i): 0\le i\le m\}$ and $t\ge \max\{d_{K_k}: k\in [p]\}$.
Assume $(L_1,\ldots,L_p)$ is an optimal solution to the program (\ref{eqxispt}) defining $\xisp_t$ and it satisfies the flatness condition: for each $k\in [p]$ there exists an integer $s_k$ such that $d_{K_k}\le s_k \le t$ and the following holds
\begin{equation}\label{eqflatk}
\rank\ L_k([x(V_k)]_{s_k}[x(V_k)]_{s_k}^T) =
\rank\ L_k([x(V_k)]_{s_k-d_{K_k}}[x(V_k)]_{s_k-d_{K_k}}^T) =:r_k.
\end{equation}
 Then, equality $\xisp_t=\valsp (=\valid)$ holds, and problem (\ref{opt:sparse}) has an optimal solution $(\mu_1,\ldots,\mu_p)$, where
 each $\mu_k$ is finite atomic and supported on $r_k$ atoms in $K_k$ for each $k\in [p]$.
\end{corollary}

Note that, for the application to the completely positive rank and the nonnegative rank, all involved polynomials in the corresponding instances of GMP are quadratic, so that $d_K=d_{K_k}=1$ and the smallest relaxation level that can be considered is $t=1$. 
For the application to the cp-rank, if the flatness condition holds for an optimal solution for the parameter $\xicpid_t(A)$  (or for the parameter $\xicpsp_t(A)$), then the parameter is equal to $\taucp(A)$ and one can extract a cp-factorization of $A$. In this way one finds an explicit factorization of $A$ and thus an upper bound on its cp-rank. In this case, if the computed value of  $\taucp(A)$ is equal to the number of recovered atoms, this certifies that  $\taucp(A)$ is equal to the cp-rank and the recovered cp-decomposition of $A$ is an optimal one. We will illustrate this on some examples in Section \ref{sec:sparsecp}. In the same way, for the application to the nonnegative rank, if the flatness condition holds for an optimal solution for the parameter $\xinnid_t(M)$ (or for the parameter $\xinnsp_t(M)$), then the parameter is equal to $\taunn(M)$ and one can extract a nonnegative factorization of $M$.

\section{Ideal-sparsity   for GMP} \label{sec:sparseGMP}

In this section we investigate how ideal-sparsity can be exploited for the GMP (\ref{opt:dense}).
First, we consider in Section \ref{sec:sparseGMPmom} the ideal-sparse reformulation (\ref{opt:sparse}) and the corresponding ideal-sparse bounds, and, after that, we mention in Section \ref{seccsp} how this relates to the more classic approach based on exploiting correlative sparsity. 

\subsection{Ideal-sparse moment relaxations} \label{sec:sparseGMPmom}

Consider the GMP (\ref{opt:dense}), where the set $K$ is defined as in (\ref{eqsetK}).
 As in Section~\ref{sec:intro}, we consider the graph $G=(V,E)$, whose maximal cliques are denoted $V_1,\ldots,V_p$, and we define the sets $\wh {K_k}\subseteq K\subseteq \R^n$ (as in (\ref{eqwhKk})) and their projections $K_k\subseteq \R^{|V_k|}$ (as in (\ref{eqKk})). Recall from (\ref{eqKkcover}) that $K=\wh{K_1}\cup\ldots\cup \wh{K_p}.$ 
 Then, one can define the (sparse) version (\ref{opt:sparse}) of GMP.
As observed above, while problem (\ref{opt:dense}) has a single  measure variable $\mu$ whose support is contained in $K\subseteq \R^n$, problem (\ref{opt:sparse}) involves $p$ measure variables $\mu_1,\ldots,\mu_p$, where $\mu_k$ is supported on the set $K_k\subseteq \R^{|V_k|}$, thus a smaller dimensional space.
We now show that both formulations (\ref{opt:dense}) and (\ref{opt:sparse}) are equivalent.

\begin{proposition}\label{propequiv}
Problems (\ref{opt:dense}) and (\ref{opt:sparse}) are equivalent, i.e., their optimum values are equal: $\valid=\valsp$.
\end{proposition}

\begin{proof} 
First, we show
$\valid \le \valsp$.  For this, assume $(\mu_1,\ldots,\mu_p)$ is feasible for problem (\ref{opt:sparse}). Consider the measure $\mu$ on $\R^{|V|}$,  defined by $\int fd\mu=\sum_{k=1}^p \int_{K_k}f_{|V_k}d\mu_k$ for any measurable function $f$ on $\R^{|V|}$. We have $\supp(\mu)\subseteq K$. Indeed,   $\int_K fd\mu= \int f \chi^Kd\mu=\sum_k \int_{K_k} f_{|V_k} \chi^K_{|V_k} d\mu_k= \sum_k \int_{K_k}f_{|V_k} d\mu_k=\int fd\mu$, since $\chi^K_{|V_k}(y)=
\chi^K(y,0_{V\setminus V_k})=1$ for all $y\in K_k$ as $(y,0_{V\setminus V_k})\in \wh {K_k}\subseteq K$. Then, $\mu$ is feasible for (\ref{opt:dense}), with the same objective value as $(\mu_1,\ldots,\mu_p)$, which shows $\valid \le \valsp$.

We now show the reverse inequality $\valsp \le \valid$. For this, assume $\mu$ is feasible for~(\ref{opt:dense}). We now define a feasible solution $(\mu_1,\ldots,\mu_p)$ to (\ref{opt:sparse}), with the same objective value as $\mu$. 
For $k\in [p]$, define the set
$$\Lambda_k=\{x\in K: \supp(x)\subseteq V_k,\ \supp(x) \not\subseteq V_h \text{ for } 1\le h\le k-1\}.$$
As each $x\in K$ has its support contained in some $V_k$, it follows that the sets $\Lambda_1,\ldots,\Lambda_p$ form a disjoint partition of $K$. 
Note that  $\Lambda_k\subseteq \wh{K_k}$ and thus $x(V_k)\in K_k$ for any $x\in \Lambda_k$.
Consider the measure $\mu_k$ on $\R^{|V_k|}$, defined by
$\int f d\mu_k= \int _{\Lambda_k} f(x(V_k))d\mu(x)$
for any measurable function $f$ on $\R^{|V_k|}$.
Then, $\supp(\mu_k)\subseteq K_k$, since $\int_{K_k}f d\mu_k=\int f \chi^{K_k}d\mu_k= \int_{\Lambda_k} f(x(V_k)) \chi^{K_k} (x(V_k))d\mu(x)= \int_{\Lambda_k}f(x(V_k))d\mu(x)=\int f d\mu_k$, as $\chi^{K_k} (x(V_k))=1$ for all $x\in \Lambda_k$.
Next, we show that $\int pd\mu=\sum_k \int p_{|V_k}d\mu_k$ for any measurable function $p:\R^{|V|}\to\R$. 
Indeed,
as  
the sets  $\Lambda_1,\ldots,\Lambda_p$  disjointly partition the set $K$, we have 
$\int pd\mu=\int_Kpd\mu=\sum_k \int_{\Lambda_k} pd\mu$. Combining with   $\int_{\Lambda_k} p(x)d\mu(x)= \int_{\Lambda_k} p_{|V_k}(x(V_k)) d\mu(x)= \int_{K_k} p_{|V_k}d\mu_k$, gives the desired identity
$\int pd\mu=\sum_k \int p_{|V_k}d\mu_k$.
Therefore, $(\mu_1,\ldots,\mu_p)$ is a feasible solution to (\ref{opt:sparse}) with the same value as $\mu$, which shows 
$\valsp \le \valid$.
\end{proof}

Based on the reformulation (\ref{opt:sparse}), we can define the {\em ideal-sparse} moment relaxation (\ref{eqxispt}) for problem (\ref{opt:dense}), which we repeat here for convenience: for any integer $t\ge 1$,
\begin{equation}\label{eqxisptb}
\begin{array}{ll}
\xisp_t:=\inf\Big \{\sum_{k=1}^p L_k({f_0}_{|V_k}): & L_k \in \R[x(V_k)]_{2t}^* \ (k\in [p]),\\
&  \sum_{k=1}^p L_k({f_i}_{|V_k})=a_i\ (i\in [N]),\\
& L_k\ge 0\text{ on } \M(\bg_{|V_k})_{2t} \ (k\in [p])\Big\}.
\end{array}
\end{equation}
This hierarchy provides bounds for $\valid$ that are   at least at good as the bounds $\xiid_t$ from (\ref{eqxiidt}).

\begin{theorem}\label{lem:dens<spar}
 For any integer $t\ge 1$ we have $\xiid_t\le \xisp_t \le \valid$. In addition, if $\M(\bg)$ is Archimedian and~(\ref{eqCQ}) holds, then $\lim_{t\to\infty} \xisp_t = \valid$.
\end{theorem}

\begin{proof}
Clearly, $\xisp_t\le \valsp$, which, combined with Proposition \ref{propequiv}, gives $\xisp_t\le \valid$.
We now show $\xiid_t\le \xisp_t$. For this, assume $(L_1,\ldots,L_p)$ is feasible for (\ref{eqxisptb}). Define $L\in \R[x]_{2t}^*$ by setting $L(p)=\sum_{k=1}^p L_k(p_{|V_k})$ for any $p\in \R[x]_{2t}$. By construction, $L(f_i)=\sum_k L_k({f_i}_{|V_k})$ for $0\le i\le m$, so that $L(f_i)=a_i$ for $i\in [m]$),  and $L\ge 0$ on $\M(\bg)$. For each $\{i,j\}\in \olE$ and $k\in [p]$,  we have $\{i,j\}\not\subseteq V_k$ and thus ${(x_ix_j)}_{|V_k}$ is identically zero; hence, for any $u\in\R[x]_{2t-2}$, we have
$L(ux_ix_j)=\sum_k L_k(u_{|V_k} {(x_ix_j)}_{|V_k})=0$.
Hence, $L$ is  feasible for (\ref{eqxiidt}) with the same objective value as $(L_1,\ldots,L_p)$, which shows $\xiid_t\le \xisp_t$. Convergence of $\xisp_t$ to $\valid$ follows from the just proven fact that $\xiid_t \le \xisp_t$ and from Theorem~\ref{theoconvGMP}, which implies $\lim_{t\to\infty}  \xiid_t = \valid$ under the stated assumptions.
\end{proof}

Observe that in Theorem \ref{lem:dens<spar} no structural chordality property needs to be assumed on the cliques $V_1,\ldots,V_p$ of the graph $G$.
In other words, the cliques $V_1,\ldots,V_p$ need not satisfy the running intersection property (see (\ref{eqVRIP}) below), which is a characterizing property of chordal graphs that is often used in  sparsity exploiting techniques like correlative sparsity.
In Section \ref{seccsp} below, we will comment about the link between the ideal-sparsity approach presented here and the more classical  correlative sparsity approach  that can be followed when considering a chordal extension $\wh G$ of the graph $G$.

\medskip

As  mentioned earlier in the introduction, the sparse bounds $\xisp_t$ present a double advantage compared to the dense bounds $\xiid_t$: they are at least as good (and often strictly  better), and their computation is potentially faster since  the sets $V_k$ can be much smaller than the full set $V$. We will see later examples illustrating this. 
On the other hand, a possible drawback is that the number of maximal cliques of $G$ could be large.   Indeed, it is well-known that the number of maximal cliques  can be exponential in the number of nodes (this is the case, e.g., when $G$ is a complete graph on $2n$ nodes with a deleted perfect matching).  A possible remedy is to consider a graph $\wt G=(V, \wt E)$ containing  $G$ as a subgraph, i.e., such that $E\subseteq \wt E$.
Then, let $\wt V_1,\ldots,\wt V_{\wt p}$ denote the maximal cliques of $\wt G$, whose number $\wt p$ satisfies $\wt p \le p$, since each maximal clique of $G$ is contained in a maximal clique  of $\wt G$. 
One can define the corresponding ideal-sparse moment hierarchy of bounds, denoted  $\wtxisp_t$, which involves $\wt p$ measure variables supported on  the sets $\wt V_1,\ldots, \wt V_{\wt p}$ (instead of the sets $V_1,\ldots, V_p$). However, as $\wt V_h$ may contain some non-edge of $G$,  one now needs to still impose an ideal condition on each linear functional $\wt L_h$ acting on $\R[x(\wt V_h)]$ ($h\in [\wt p]$). Namely, the parameter 
$\wtxisp_t$ is defined as
\begin{equation} \label{eqwtxispt}
\begin{array}{ll}
\wtxisp_t:=\inf\Big \{\sum_{h=1}^{\wt p}  \wt L_h({f_0}_{|\wt V_h}): &\wt L_h \in \R[x(\wt V_h)]_{2t}^* \ (h\in [\wt p]),\\
&  \sum_{h=1}^{\wt p} \wt L_h({f_i}_{|\wt V_h})=a_i\ (i\in [N]),\\
& \wt L_h\ge 0\text{ on } \M(\bg_{|\wt V_h})_{2t} \ (h\in [\wt p]),\\
& \wt L_h(x_ix_jx^\alpha)=0 \ (\alpha \in \N^n_{2t-2},~\supp(\alpha)\subseteq \wt V_h,\ \{i,j\}\subseteq \wt V_h,\ \{i,j\}\in \olE) \Big\}.
\end{array}
\end{equation}
Note that this parameter interpolates between the dense and sparse parameters: indeed, $\wtxisp_t=\xisp_t$ if $\wt G=G$, and $\wtxisp_t=\xiid_t$ if $\wt G=K_n$ is the complete graph.
Accordingly, we have the following inequalities among the parameters.

\begin{lemma}
Assume $\wt G$ contains $G$ as a subgraph. For any integer $t\ge 1$ we
have $\xiid_t\le \wtxisp_t\le \xisp_t$.
\end{lemma}

\begin{proof}
The proof for the inequality $\xiid_t\le \wtxisp_t$ is analogous to the proof of $\xiid_t\le \xisp_t$ in Theorem \ref{lem:dens<spar}. We now show $\wtxisp_t\le \xisp_t$. For this, assume $(L_1,\ldots,L_p)$ is feasible for the parameter $\xisp_t$. As each clique $V_k$ of $G$ is contained in some clique $\wt V_h$ of $\wt G$, there exists a  partition $[p]= A_1\cup \ldots\cup A_{\wt p}$ such that 
$V_k\subseteq \wt V_h$ for all $k\in A_h$ and $h\in [\wt p]$.
For $h\in [\wt p]$, we define $\wt L_h\in \R[x(\wt V_h)]_{2t}^*$ by setting 
$\wt L_h(p)= \sum_{k\in A_h} L_k(p_{|V_k})$ for $p\in \R[x(\wt V_h)]_{2t}$.
Then, one can easily verify that $(\wt L_1,\ldots,\wt L_{\wt p})$ provides a feasible solution for $\wtxisp_t$, with the same objective value as $(L_1,\ldots,L_p)$. Let us only check the ideal constraint. For this assume  $\{i,j\}\cup \supp(\alpha)\subseteq \wt V_h$ and $\{i,j\}\in\olE$. Then, $\{i,j\}$ is not contained in any clique $V_k$ of $G$ and thus $L_k((x_ix_jx^\alpha)_{|V_k})=0$ for all $k\in [p]$, which directly implies $\wt L_h(x_ix_jx^\alpha)=0$.
\end{proof}

\begin{remark}\label{remsparsechordal}
There may be  a trade-off to be made between  the parameter $\xisp_t$, which fully exploits  the sparsity of $G$ (and provides a possibly better bound), and  the parameter $\wtxisp_t$, which only partially exploits the sparsity, depending on the choice of the extension $\wt G$ of $G$. 
Namely, the parameter  $\xisp_t$ may involve many cliques of smaller sizes, while the parameter $\wtxisp_t$ involves less cliques but with larger sizes. 
If one cares to have a small number of cliques, then one can (but is not required to) consider  for $\wt G$ a chordal extension $\wh G$ of $G$, in which case the number of maximal cliques is at most the number of nodes.

In our numerical experiments for matrix factorization ranks we will consider only the two extreme cases of the dense and ideal-sparse parameters $\xiid_t$ and $\xisp_t$. For most of the matrices considered  the number of maximal cliques seems  indeed not to play a significant role. However, when this number becomes too large, one may have to consider alternative intermediate parameters (see Section \ref{sec:final} for a brief discussion).
\end{remark}

\subsection{Bounds based on correlative sparsity}\label{seccsp}

In this section we compare the ideal-sparse approach with the more classic one based on exploiting correlative sparsity.
The setting of correlative sparsity is usually applied to a polynomial optimization problem, where each of the polynomials arising as a constraint involves only a subset of the variables (indexed, say,
by one of the subsets $\wV_1,\ldots,\wV_{\widehat p}$) and the objective polynomial is a sum of such polynomials. Then, one can define more economical relaxations that respect this sparsity pattern. In the case when the sets $\wV_1,\ldots,\wV_{\widehat p}$ respect the so-called RIP property (see (\ref{eqVRIP}) below) (and under some Archimedean condition), these hierarchies enjoy asymptotic convergence properties analogous to the dense hierarchies; see \cite{lasserre2006convergent,GrimmNS} for details and also \cite{sparsebook} for general background on correlative sparsity.
We now explain  how correlative sparsity applies to the instance of GMP considered in this paper.

\medskip
As before,  we assume $K$ is contained in the variety of the ideal $I_E$, generated by the monomials $x_ix_j$ corresponding to the nonedges of the graph $G=(V,E)$. In the ideal-sparsity approach we considered a measure variable for each maximal clique of $G$. However, the number  of maximal cliques of $G$ can be large, which could represent a drawback for this approach. 

An alternative is to consider a {\em chordal extension} $\wh G=(V, \wh E)$ of $G$, that is, a chordal graph $\wh G$ containing $G$ as a subgraph, i.e., such that  $E\subseteq \wh E$. 
Then, as a well-known property of chordal graphs,   $\wh G$ has at most $n$ distinct maximal cliques. Let
$\wV_1,\ldots,\wV_\whp$ denote the maximal cliques of $\wG$, so $\whp\le n$.
As one of the many equivalent definitions of chordal graphs, it is known that  the maximal cliques $\wV_1,\ldots,\wV_\whp$ satisfy (possibly after  reordering) the 
 so-called {\em running intersection property} (RIP):
\begin{equation}\label{eqVRIP} 
\forall k\in \{2,\ldots, \whp\}\ \ \exists j\in \{1,\ldots,k-1\}\ \ \text{ such that }\ \ \wV_k\cap(\wV_1\cup\ldots \cup \wV_{k-1})\subseteq \wV_j.
\end{equation}
See, e.g., \cite{D17} for details. As we explain below, it turns out that one can `transport' the chordal sparsity structure of the graph $\wG$ to the moment matrices involved in the definition of  the  dense bound $\xiid_t$ in (\ref{eqxiidt}). 

\medskip
To see this, let us first    rewrite the parameter $\xiid_t$  more concretely  as a semidefinite program.
For convenience, set $d_j:=\lceil \deg(g_j)/2\rceil$ for $j\in [m]$. Then, following the discussion in Section \ref{sec:prelmom},
 the parameter $\xiid_t$ can be expressed as
\begin{equation}\label{eqxiidb}
\begin{array}{ll}
\xiid_t= \inf\{L(f_0): & L\in\R[x]_{2t}^*,\ L(f_i)=a_i\ (i\in [N]),\\
& L([x]_t[x]_t^T)\succeq 0,\ L(g_j[x]_{t-d_j}[x]_{t-d_j}^T)\succeq 0 \ (j\in [m]),\\
& 
L=0 \text{ on } I_{E,2t}, \text{ i.e.,}\ L(x_ix_jx^\alpha)=0 \ (\{i,j\}\in\olE,\  \alpha\in \N^{n}_{2t-2})\}.
\end{array}
\end{equation}
For fixed  $t\in \N$, define the  sets
\begin{equation}\label{eqMIt}
\MI_{k,t}=\{\alpha \in \N^n_t: \supp(\alpha)\subseteq \wV_k\}\subseteq \N^n_t \  (k\in [\whp]),  \quad  
\MI_t=\bigcup_{k=1}^\whp \MI_{k,t}\subseteq \N^n_t.
\end{equation}

\begin{lemma}\label{lemRIP0}
Assume  $L\in \R[x]_{2t}^*$ satisfies  $L=0$ on $I_{E,2t}$. Then, 
 $L(x^\alpha x^\beta)=0$ for any $\alpha, \beta \in \N^n_t$ such that $\{\alpha,\beta\}$ is not contained in any of the sets $\MI_{1,t},\ldots,\MI_{\whp,t}$. 
\end{lemma}
\begin{proof}
Assume  there is no index $k\in [\wh p]$ such that $\{\alpha,\beta\}\subseteq \MI_{k,t}$. Then, $\supp(\alpha+\beta)$ is not a clique in $G$, for otherwise  it would be contained in some $\wV_k$, implying $\supp(\alpha),\supp(\beta)\subseteq \wV_k$ and thus $\alpha,\beta\in \MI_{k,t}$, yielding a contradiction. As $\supp(\alpha+\beta)$ is not a clique in $G$, it contains a pair $\{i,j\}\in \olE$, which implies $x^\alpha x^\beta\in I_{E,2t}$ and thus $L(x^\alpha x^\beta)=0$.
\end{proof}

In view of Lemma \ref{lemRIP0}, in the definition of $\xiid_t$ in (\ref{eqxiidb}),  one may restrict the  matrix $L([x]_t[x]_t^T)$ to its principal submatrix indexed by $\MI_t$, since any  row/column indexed by $\alpha\in \N^n_t\setminus \MI_t$ is identically zero. Moreover, $L(x^\alpha x^\beta)\ne 0$ implies $\{\alpha,\beta\}\subseteq \MI_{k,t}$ for some $k\in [\whp]$. In other words,  the support graph of the matrix $L([x]_t[x]_t^T)$ is contained in the graph  with vertex set $\MI_t$, whose  maximal cliques are the sets $\MI_{1,t},\ldots,\MI_{\whp,t}$.
The next lemma shows that the RIP property also holds for the sets $\MI_{1,t},\ldots, \MI_{\whp,t}$. 
Therefore,  the moment matrix $M_t(L)=L([x]_t[x]_t^T)$ has a correlative sparsity pattern, which it inherits from  the chordal extension $\wh G$ of $G$.

\begin{lemma}\label{lemRIPMI}
The sets $\MI_{1,t},\ldots,\MI_{\whp,t}$ satisfy the RIP property:
\begin{align}\label{RIPMI}
\forall q\in \{2,\ldots, \whp\}\ \ \exists k\in \{1,\ldots,q-1\}\ \ \text{ such that }\ \ 
\MI_{q,t} \cap(\MI_{1,t} \cup\ldots \cup \MI_{q-1,t})\subseteq \MI_{k,t}.
\end{align}
\end{lemma}

\begin{proof}
Let $q\in \{2,\ldots,\whp\}$ and assume by way of contradiction that there exists no $k\in [q-1]$ for which $\MI_{q,t}\cap (\MI_{1,t}\cup \ldots \cup \MI_{q-1,t})\subseteq \MI_k$ holds.
Then, for each $k\in [q-1]$, there exists $\alpha^k\in \MI_{q,t}\cap (\MI_{1,t}\cup \ldots \cup \MI_{q-1,t})\setminus \MI_{k,t}$ and thus there exists $i_k\in V\setminus \wh V_k$ such that $\alpha^k_{i_k}\ge 1$. 
As $\alpha^k\in \MI_{q,t}$ and $\alpha^k_{i_k}\ge 1$, it follows that $i_k\in \wh V_q$.
In addition, $\alpha^k\in \MI_{j,t}$ for some $j\in [q-1]$. Again, as $\alpha^k_{i_k}\ge 1$, it follows that $i_k\in \wh V_j$. This shows that $$i_k\in \wh V_q\cap(\wh V_1\cup \ldots \cup \wh V_{q-1}) \quad \text{ for all } k\in [q-1].$$
By the RIP property (\ref{eqVRIP}) for  $\wh V_1,\ldots,\wh V_p$, there exists $q_0\in [q-1]$ such that $\wh V_q\cap(\wh V_1\cup\ldots \cup \wh V_{q-1})\subseteq \wh V_{q_0}$. Therefore, $i_k\in \wh V_{q_0}$ for all $k\in [q-1]$.
As $i_k\not\in \wh V_k$, this implies that $q_0\ne k$ for all $k\in [q-1]$, and we reach a contradiction.
\end{proof}

The above extends  to the localizing matrices $L(g_j[x]_{t-d_j}[x]_{t-d_j}^T)$ for $j\in [m]$.
In the same way, one may restrict the  matrix $L(g_j[x]_{t-d_j}[x]_{t-d_j}^T)$ to its principal submatrix indexed by $\MI_{t-d_j}$ and its support graph is contained in the graph with vertex set  $\MI_{t-d_j}$, whose  maximal cliques are the sets $\MI_{1,t-d_j},\ldots,\MI_{\whp,t-d_j}$.
Moreover,  there is a correlative sparsity pattern on the matrix $L(g_j[x]_{t-d_j}[x]_{t-d_j}^T)$ ($0\le j\le m$), which is inherited from the chordal structure of $\wG$.

\medskip
Therefore, one may apply Theorem \ref{theoAHMCR} below to get a more economical reformulation of $\xiid_t$.
Indeed, by Theorem \ref{theoAHMCR}, 
one may write $ L(g_j[x]_{t-d_j}[x]_{t-d_j}^T) =\sum_{k=1}^\whp Z_{j,k}$, where $Z_{j,k}$ is obtained from a matrix  indexed by the set $\MI_{k,t-d_j}$ by padding it with zero entries, and replace the condition $ L(g_j[x]_{t-d_j}[x]_{t-d_j}^T) \succeq 0$ by the conditions $Z_{j,1},\ldots,Z_{j,\whp}\succeq 0$. 
The advantage is that requiring $Z_{j,k}\succeq 0$ boils down to checking positive semidefiniteness of   a potentially much smaller matrix, indexed by $\MI_{k,t-d_j}$. Hence, this allows to replace 
one (large) positive semidefinite matrix by several smaller positive semidefinite matrices.
While this method offers a more economical way for computing the dense parameter $\xiid_t$, it is nevertheless inferior to the ideal-sparse approach described in the previous section. Recall in particular Remark \ref{remsparsechordal}, where we indicated how to construct a sparse parameter $\wtxisp_t$,
which could also be based on using a chordal extension $\wh G$ of $G$, but superior in quality as $\xiid_t\le \wtxisp_t$.

\begin{theorem}[\cite{AHMcCR}]\label{theoAHMCR}
Consider a positive semidefinite matrix $X\in \MS^n_+$ whose  support graph is contained in a chordal graph $\wG$, with maximal cliques $\wV_1,\ldots,\wV_\whp$. Then, there exist positive semidefinite matrices $Y_k\in \MS^{\wV_k}_+$ ($k\in [\whp]$) such that $X=\sum_{k=1}^\whp Z_k$, where $Z_k=Y_k\oplus 0_{V\setminus \wV_k,V\setminus \wV_k}\in\MS^n_+$ is obtained by padding $Y_k$ with zeros.
\end{theorem}

\medskip
As a final observation,  another possibility to exploit the above correlative sparsity structure would be simply  to  replace in the definition of $\xiid_t$ in program (\ref{eqxiidt}) each  condition $L(g_j[x]_{t-d_j}[x]_{t-d_j})\succeq 0$ by $\whp$ smaller matrix conditions $L({g_j}_{|\wV_k}[x(\wV_k)]_{t-d_j}[x(\wV_k)]_{t-d_j})\succeq 0$ for $k\in [\whp]$.
In other words, if $L_{|V_k}$ denotes the restriction of $L$ to the polynomials in variables indexed by $\wV_k$, then we replace the condition $L\ge 0$ on $\M(\bg)_{2t}$ by the conditions  $L_{|\wV_k}\ge 0$ on $\M(\bg_{|\wV_k})_{2t}$ for each $k\in[\whp]$. In this way we obtain another parameter, denoted by $\xi^{\text{\rm csp}}_t$, that is weaker than $\xiid_t$ and thus satisfies
$$\xi^{\text{\rm csp}}_t\le \xiid_t \le \wtxisp_t \le\xisp_t.$$
Recall $\wtxisp_t$ is the parameter from (\ref{eqwtxispt}) obtained when selecting an  extension $\wt G$ of $G$, including, for instance,  selecting a chordal extension $\wt G=\wh G$.

\section{Application to the completely positive rank} \label{sec:app-cp}

In this section we investigate how ideal-sparsity can be exploited to design bounds on the completely positive rank. 
We define the corresponding hierarchies  of lower bounds on the cp-rank and indicate their relations to other known bounds in the literature.

\subsection{Ideal-sparse lower bounds on the cp-rank}\label{seccp1}

Consider a symmetric nonnegative matrix $A\in \mathcal S^n$ and assume $A_{ii}\ne 0$ for all $i\in V$ (to avoid trivialities).
Then, its cp-rank, denoted  $\cprank(A)$, is the  smallest integer $r\in\N$ for which $A$ admits a decomposition of the form $A=\sum_{\ell=1}^ra_\ell a_\ell^T$ with $a_\ell \ge 0$ (setting $r=\infty$ if no such decomposition exists, when $A$ is not completely positive). 
Fawzi and Parrilo \cite{FP16} introduced the parameter $\taucp(A)$ from (\ref{eqtaucp}), as a convexification of the cp-rank, whose definition is repeated  for convenience:
 \begin{equation*}\label{eqtaucpb}
 \taucp(A):=\min\Big\{\lambda: {1\over \lambda }A \in \text{conv}\{xx^T: x\in \R^n_+,\ A-xx^T\succeq 0, A\ge xx^T\}\Big\}.
\end{equation*}
Clearly, we have  $\taucp(A)\le \cprank(A)$. 
 As was already indicated in Section \ref{sec:intro},
the parameter $\taucp(A)$ can be reformulated as an instance of problem (\ref{opt:dense}) with an ideal-sparsity structure inherited from the matrix $A$. For this, recall $G_A=(V=[n],E_A)$ denotes the support graph of $A$, where $E_A$ consists of all pairs  $\{i,j\}$ with $i\ne j\in V$ and $A_{ij}\ne 0$ (as in (\ref{eqsetEA})), and recall the definition of the semialgebraic set $K_A$ from (\ref{eqsetKA}).
As shown in Lemma \ref{lemtaucpGMP}, $\taucp(A)$ can be reformulated as an instance of GMP: 
$$\taucp(A)=\inf_{\mu\in \meas(\R^n)}\Big\{\int_{K_A}1 d\mu: \int_{K_A} x_ix_jd\mu = A_{ij}\ (i,j\in V), \ \supp(\mu)\subseteq K_A\Big\}.$$

\medskip
\paragraph{Dense hierarchies for cp-rank.} Based on the above reformulation of $\taucp(A)$,   for any integer $t\ge 1$,  let us define the following parameter
(as special instance of  (\ref{eqxiidt})):
  \begin{align}\xicpid_t(A)= \min\{L(1): & \ L\in \R[x]^*_{2t},\\
&   L(x_ix_j)=A_{ij}\ (i,j\in V), \label{eqcpA} \\
& L([x]_t[x]_t^T)\succeq 0, \label{eqcp0}\\
&L((\sqrt{A_{ii}}x_i-x_i^2)[x]_{t-1}[x]_{t-1}^T)\succeq 0 \ \text{ for } i\in V, \label{eqcpxi}\\
&L((A_{ij}-x_ix_j)[x]_{t-1}[x]_{t-1}^T)\succeq 0 \ \text{ for }\{i,j\}\in E_A,\label{eqcpxixj}\\
&L(x_ix_j[x]_{2t-2})=0\ \text{ for } \{i,j\}\in \olE_A, \label{eqcpid}\\
& L( {(A-xx^T)}\otimes [x]_{t-1}[x]_{t-1}^T)\succeq 0.\label{eqcpmat} 
\end{align}

We first  indicate how this parameter relates to other similar moment-based bounds considered in the literature, in particular in \cite{GdLL2019a} and \cite{GLS2022}.
Note that, due to the presence of the (ideal) constraints (\ref{eqcpid}), the constraint (\ref{eqcpxixj}) trivially holds for any pair $\{i,j\}\in \olE_A$.
If we omit the ideal constraint (\ref{eqcpid}) and impose the constraint (\ref{eqcpxixj}) for {\em all} pairs $\{i,j\}$ with $i\ne j\in V$,  then we obtain a parameter  investigated in \cite{GLS2022},  denoted here as $\xicp_{t,(2022)}(A)$.
The parameter $\xicp_{t,(2022)}(A)$ strengthens an earlier parameter $\xicp_{t,(2019)}(A)$ introduced in \cite{GdLL2019a}, whose definition follows by replacing in the definition of $\xicp_{t,(2022)}(A)$ the constraint (\ref{eqcpmat}) by the weaker constraint
\begin{align}\label{eqcpmatll}
L((xx^T)^{\otimes \ell})\preceq A^{\otimes \ell}\ \text{ for } \ell \in [t].
\end{align}
So, for any $t\ge 1$,  we have 
$$\xicp_{t,(2019)}(A)\le \xicp_{t,(2022)}(A)\le \xicpid_t(A).$$ 
Since the  bounds $\xicp_{t,(2019)}(A)$ were shown to converge asymptotically to $\taucp(A)$ in \cite{GdLL2019a}, the same holds for the bounds  $\xicpid_t(A)$.
Note that the convergence of the latter bounds  also  follows directly from Theorem \ref{theoconvGMP}. 

As mentioned in \cite{GdLL2019a}, there are more constraints that can be added to the above program and still lead to a lower bound on the cp-rank (in fact on $\taucp(A)$). In particular,  exploiting the fact that the variables $x_i$ should be nonnegative, one may add the constraints
\begin{align}
 L([x]_{2t}) \ge 0,\label{cpdagger1}  \\
 L((\sqrt{A_{ii}}x_i-x_i^2)[x]_{2t-2})\ge 0 \text{ for } i\in V, \label{cpdagger2} \\ 
 L(A_{ij}-x_ix_j)[x]_{2t-2}) \geq 0 \text{ for }\{i,j\}\in E_A.\label{cpdagger3} 
 \end{align}
One may also add other localizing constraints, such as
\begin{align}\label{eqcpxx}
L(x_ix_j[x]_{t-1}[x]_{t-1}^T) \succeq 0 \text{ for }\{i,j\}\in E_A.
\end{align}
Note that the constraints (\ref{eqcpxx}) are redundant at the smallest level $t=1$. Note also that one could add a similar constraint replacing $x_ix_j$ by any monomial.
We  use  the notation $\xicpid_{t,\dag}(A)$ to denote the parameter obtained by adding
(\ref{cpdagger3}) to the program defining $\xicpid_t(A)$. 
Define analogously $\xicp_{t,(2019),\dag}(A)$ by adding (\ref{cpdagger3}) to $\xicp_{t,(2019)}(A)$, 
so that we have
\begin{align*}\label{eqrelcpdag}
\xicp_{t,(2019),\dag}(A)\le \xicpid_{t,\dag}(A).
\end{align*}
As we will see in relation (\ref{eqrelcp}) below, the bound $\xicp_{2,(2019),\dag}(A)$ is at least as good as $\rank (A)$, an obvious lower bound on $\cprank(A)$.
Let $\xicpid_{t,\ddagger}(A)$ denote the strengthening of $\xicpid_{t,\dagger}(A)$ by adding constraints (\ref{cpdagger1}), (\ref{cpdagger2}), and (\ref{eqcpxx}), so that we have
$\xicpid_t(A)\le \xicpid_{t,\dag}(A)\le \xicpid_{t,\ddagger}(A)$.

\medskip
\paragraph{Ideal-sparse hierarchies for cp-rank.} 
We now consider the ideal-sparse bounds for the cp-rank,  which further exploit the ideal-sparsity pattern of $A$. For this, let $V_1,\ldots,V_p$ denote the maximal cliques of the graph $G_A$ and, for $t\ge 1$, define the following parameter (as special instance of (\ref{eqxispt})):
  \begin{align}
  \xicpsp_t(A)= \min\Big\{\sum_{k=1}^pL_k(1): & \ L_k\in \R[x(V_k)]^*_{2t}\ (k\in [p]),\\
&   \sum_{k\in [p]: i,j\in V_k}  L_k(x_ix_j)=A_{ij}\ (i,j\in V), \label{eqcpAsp} \\
& L_k([x(V_k)]_t[x(V_k)]_t^T)\succeq 0 \ (k\in [p]), \label{eqcp0sp}\\
&L_k((\sqrt{A_{ii}}x_i-x_i^2)[x(V_k)]_{t-1}[x(V_k)]_{t-1}^T)\succeq 0 \ \text{ for } i\in V_k,\ k\in [p], \label{eqcpxisp}\\
&L_k((A_{ij}-x_ix_j)[x(V_k)]_{t-1}[x(V_k)]_{t-1}^T)\succeq 0 \ \text{ for } i\ne j\in V_k,\ k\in [p], \label{eqcpxixjsp}\\
& L_k( {(A- x x^T)} \otimes   [x(V_k)]_{t-1}[x(V_k)]_{t-1}^T)\succeq 0,\ \text{ for } k\in [p]. \label{eqcpmatsp} 
\end{align}  
Here, in equation (\ref{eqcpmatsp}), it is understood that, for a given $k\in [p]$,  in the matrix $A-xx^T$  one sets the entries of $x$  indexed by $V\setminus V_k$ to zero.
As a direct application of Theorem \ref{lem:dens<spar}, we have
$$
\xicp_t(A)\le \xicpsp_t(A)\le \taucp(A) \text{ for any } t\ge 1.
$$
One may also define the sparse analogs of the constraints (\ref{cpdagger1}), (\ref{cpdagger2}), (\ref{cpdagger3}), and (\ref{eqcpxx}):
\begin{align}
 L_k([x(V_k)]_{2t}) \ge 0 \ \text{ for } k\in [p] ,\label{cpdaggersp1} \\
  L_k((\sqrt{A_{ii}}x_i-x_i^2)[x(V_k)]_{2t-2})\ge 0\ \text{ for } i\in V_k,\ k\in [p],\label{cpdaggersp2}\\ 
 L_k((A_{ij}-x_ix_j)[x(V_k)]_{2t-2}) \geq 0\ \text{ for }\{i,j\}\subseteq V_k,\ k\in [p],\label{cpdaggersp3}\\ 
 L_k(x_ix_j[x(V_k)]_{t-1}[x(V_k)]_{t-1}^T) \succeq 0\ \text{ for }i\ne j\in V_k,\ k\in [p]. \label{eqcpxxsp} 
\end{align}
Then, define $\xicpsp_{t,\dag}(A)$ by adding constraint (\ref{cpdaggersp3}) to $\xicpsp_t(A)$, and
$\xicpsp_{t,\ddagger}(A)$  
by adding the constraints (\ref{cpdaggersp1}), (\ref{cpdaggersp2}) and (\ref{eqcpxxsp}) to
$\xicpsp_{t,\dag}(A)$, so that $\xicpsp_t(A)\le \xicpsp_{t,\dag}(A)\le \xicpsp_{t,\ddagger}(A)$.

\paragraph{Weak ideal-sparse hierarchies for cp-rank.} 
Observe that, if, in equation (\ref{eqcpmatsp}), we replace the matrix $A-xx^T$ by its principal submatrix indexed by $V_k$, then  one also gets a lower bound on $\taucp(A)$, possibly weaker than $\xicpsp_t(A)$, but potentially easier to compute. 
Let $\xicpwsp_t(A)$ denote the parameter obtained  by replacing the condition (\ref{eqcpmatsp}) in the definition of $\xicpsp_t(A)$  by the following (weaker) constraint
\begin{align}\label{eqcpmatwsp}
L_k( {(A[V_k]- x(V_k)x(V_k)^T)} \otimes   [x(V_k)]_{t-1}[x(V_k)]_{t-1}^T)\succeq 0\ \text{ for } k\in [p].
\end{align}
Then we have 
 $$\xicpwsp_t(A)\le \xicpsp_t(A).$$
Since we have weakened some conditions of the ideal-sparse hierarchy, the weak ideal-sparse hierarchy $\xicpwsp_t(A)$ is no longer guaranteed to be at least as strong as  the dense hierarchy  $\xicp_t(A)$. 
This is substantiated by our numerical experiments, where we frequently observe $\xicpwsp_t(A) < \xicp_t(A)$ for randomly generated matrices $A$; see Section \ref{randgen} for how we generate these matrices and see (\ref{eqA}) for a concrete instance of such matrix. 
On the other hand, in all of the high cp-rank matrices $A$ from the literature that we consider in Section \ref{sec:sparsecp}, it does hold that  $\xicp_t(A) \leq \xicpwsp_t(A)$. 
This relation also holds for several other cp-rank matrices from the literature we have considered but did not present in this paper. 
It  
seems  that the delineating factor might be that our randomly generated matrices tend to have cp-rank close to the usual matrix rank (i.e., $\cprank(A) - \rank(A) \leq 1$), while, in contrast, the matrices considered in the literature have a cp-rank often much higher (e.g., up to 27 for Example ex4 in (\ref{figex2})) than the rank.

\subsection{Links to combinatorial  lower bounds on the cp-rank}\label{seccp2}

We indicate here some links to other known lower bounds on the cp-rank.
Clearly the  rank is a lower bound:  $$\cprank(A)\ge \rank(A).$$
A  combinatorial lower bound  arises naturally from the edge clique-cover number of the support graph $G_A$.

Given a graph $G=(V,E)$, its {\em edge clique-cover number}, denoted $c(G)$ (following \cite{FP16}),  is defined as the smallest number of (maximal) cliques in $G$
whose union covers every edge of $G$. 
This parameter is NP-hard to compute \cite{GJ}.  
Clearly, $\ecc(G)=|E|$ if $G$ is a triangle-free graph (i.e., $\omega(G)=2$, where $\omega(G)$ denotes the maximum cardinality of a clique in $G$). 
As observed in \cite{FP16},  the edge clique-cover parameter gives a lower bound on the cp-rank:
$$\cprank(A)\ge \ecc(G_A). $$ 
Indeed, if $A=\sum_{\ell=1}^r a_\ell a_\ell^T$ with $a_\ell\ge 0$ and $r=\cprank(A)$, then the supports of $a_1,\ldots, a_r$ are (not necessarily distinct) cliques that provide an edge clique-cover of $G_A$ by at most $r$ cliques.

In \cite{FP16} a semidefinite parameter $\taucpsos(A)$ is introduced, which is shown
to be at least as good as $\rank (A)$ and as $c_{\text{\rm frac}}(G_A)$, the {\em fractional edge clique-cover number}, i.e., the natural linear relaxation of $c(G_A)$ defined by
\begin{equation}\label{eqcfrac}c_{\text{\rm frac }}(G_A)=\min\Big\{\sum_{k=1}^p x_k: \sum_{k: \{i,j\}\subseteq V_k} x_k\ge 1\ \text{ for } \{i,j\}\in E_A\Big\}.
\end{equation}
So, we have $c(G_A)\ge c_{\text{\rm frac }}(G_A)$ and 
\begin{align*}
\taucp(A)\ge \taucpsos(A) \ge \max\{\rank(A), c_{\text{\rm frac}}(G_A)\}. 
\end{align*}
 In \cite{GdLL2019a} it is shown
  that the bound $\xicp_{2,(2019),\dag}(A)$ is at least as strong as $\taucpsos(A)$. Indeed,   the proof for the relevant result (Proposition 7 in \cite{GdLL2019a}) only uses the relation $L((A_{ij}-x_ix_j)x_ix_j)\ge 0$ from (\ref{cpdagger3}) and the relation $L((xx^T)^{\otimes 2}\preceq A^{\otimes 2}$ in (\ref{eqcpmatll}).
 Hence, we have  the chain of inequalities
\begin{align}\label{eqrelcp}
\taucp(A)\ge \xicpsp_{2,\dag}(A) \ge \xicpid_{2,\dag}(A)\ge \xicp_{2,(2019),\dag}(A) \ge \taucpsos(A) \ge \max\{\rank(A), c_{\text{\rm frac}}(G_A)\}. 
\end{align}
 
As we now observe,  the (weak) ideal-sparse bound $\xicpwsp_1(A)$ of the first level $t=1$ is at least as good as the parameter $c_{\text{\rm frac }}(G_A)$.

\begin{lemma}\label{lemcplbecc}
If $A\in \mathcal S^n$ is nonnegative  with support graph $G_A$, then 
$\xicpwsp_1(A)\ge c_{\text{\rm frac }}(G_A).$  
\end{lemma}

\begin{proof}
Let $(L_1,\ldots,L_p)$ be an optimal solution  for the parameter $\xicpwsp_1(A)$. Using (\ref{eqcpxixjsp}), we have
$$L_k(A_{ij}-x_ix_j)\ge 0 \ \text{ for all } i\ne j \text{ with } \{i,j\}\subseteq V_k \text{ and } k\in [p],$$
which gives $A_{ij}L_k(1)\ge L_k(x_ix_j)$. Summing over $k$, we get
$$A_{ij}\sum_{k\in [p]: \{i,j\}\subseteq V_k} L_k(1) \ge \sum_{k\in [p]: \{i,j\}\subseteq V_k}L_k(x_ix_j)= A_{ij},
$$
where we use (\ref{eqcpAsp}) for the last equality. As $A_{ij}>0$, this gives 
$\sum_{k: \{i,j\}\subseteq V_k}L_k(1)\ge 1$ for every edge $\{i,j\}\in E_A$.
Hence, the vector $x=(L_k(1))_{k=1}^p \in \R^p_+$ is feasible for program (\ref{eqcfrac}), which implies the inequality $\sum_{k=1}^pL_k(1)\ge c_{\text{\rm frac}}(G_A)$, as desired.
\end{proof}

We now give a class of cp-matrices that exhibit a large separation between the dense and ideal-sparse bounds at level $t=1$: these matrices have size $n=2m$ and
$\cprank(A)=\xicpwsp_1(A)=m^2\ge  m+1 > \xicpid_1(A)$.

\begin{example}\label{excp}
For $n=2m$ consider the matrix 
$$A=\left(\begin{matrix} (m+1)I_m & J_m\cr J_m & (m+1)I_m\end{matrix}\right)\in\mathcal S^n,$$ 
where $I_m$ is the identity matrix and $J_m$ the all-ones matrix.
Then, $A$ is a cp-matrix (because it is nonnegative and diagonally dominant). 
Its cp-rank is $\cprank(A)=|E_A|=m^2$ (because 
 its support graph $G_A$ is the complete bipartite graph $K_{m,m}$ (thus, connected, triangle-free and not a tree),  using a result of \cite{DJL94}, also mentioned below). 
 Clearly, we have $c(G_A)=c_{\text{\rm frac}}(G_A)=m^2$.
 Hence, using Lemma \ref{lemcplbecc}, we obtain $\xicpsp_1(A)=\xicpwsp_1(A)= m^2=\cprank(A)$. We claim  $\xicpid_1(A)<m+1$, which shows a large separation between the dense and ideal-sparse bounds of level $t=1$. 

For this, observe that $\xicpid_1(A)$ can be reformulated as 
\begin{align*}\label{eqsol}
\xicpid_1(A)=\min\{L(1): L\in \R[x]_2^*, L(1)\ge 1,\ L(x_i)\ge \sqrt{A_{ii}}\ (i\in [n]), \ L(xx^T)=A,\ L([x]_1[x]_1^T)\succeq 0\}.
\end{align*}
Consider the linear functional $L\in \R[x]_2^*$ defined by 
$L(xx^T)=A$, $L(x_i)=\sqrt{m+1}$ for $i\in [n]$ and $L(1)= {2m(m+1)\over 2m+1}$. We show  that $L$ is feasible for the above program, which implies $\xicpid_1(A)\le L(1)<m+1$.  For this, it suffices to show that $L([x]_1[x]_1^T)\succeq 0$. By taking the Schur complement with respect to the upper left corner, this boils down to checking that
$ L(1) A - (m+1)J_m\succeq 0$. As the all-ones vector is an eigenvector of $A$ (with eigenvalue $2m+1$), it is also an eigenvector of $ L(1) A - (m+1)J_m$ with corresponding eigenvalue $L(1) (2m+1)- 2m(m+1)=0$. 
The eigenvalues of the matrix $L(1)A-(m+1)J_m$ for its eigenvectors orthogonal to the all-ones vector are eigenvalues of $A$, and thus they are  nonnegative since $A$ is positive semidefinite. This shows that  $ L(1) A - (m+1)J_m\succeq 0$ and the proof is complete.
\end{example}

We conclude with some observations about known upper bounds on the cp-rank.

General upper bounds on the cp-rank are  $\cprank(A)\le n$ if $n\le 4$, $\cprank (A)\le {n+1\choose 2}- 4$ if $n\ge 5$ \cite{SMBJS2013}, and $\cprank(A)\le {r+1\choose 2}-1$ if $r=\rank(A)\ge 2$ \cite{BB2003}.

It is known that $\ecc(G)\le n^2/4$ \cite{EGP}. In analogy, it has been a long standing conjecture by Drew et al. \cite{DJL94} that the cp-rank of an $n\times n$ completely positive matrix is at most $n^2/4$. This conjecture, however, was disproved in \cite{BSU2014,BSU2015} for any $n\ge 7$.  In particular, it is shown in \cite{BSU2015} that the maximum cp-rank of an $n\times n$ cp-matrix is of the order $n^2/2 + O(n^{3/2})$.

If  the support graph $G_A$ is triangle-free, then  $|E_A|\le \cprank(A)\le \max\{n, |E_A|\}$; moreover, if $G_A$ is  connected, triangle-free and not a tree, then  $\cprank(A)=|E_A|$  \cite{DJL94}.  Hence,  $n-1=|E_A|\le \cprank(A)\le n$ if $G_A$ is a tree, with $\cprank(A)=n$ if $A$ is nonsingular. 
By Lemma \ref{lemcplbecc}, we know that $\xicpwsp_1(A)\ge |E_A|$ if $G_A$ is triangle-free. 
Hence, the bound $\xicpwsp_1(A)$ gives the exact value of the cp-rank when $G_A$ is connected, triangle-free and not a tree. 
On the other hand, if $G_A$ is a tree and $A$ is nonsingular, then 
the bound $\xicpid_{2,\dag}(A)$ gives the exact value (equal to $n$) of the cp-rank (since it is at least $\taucpsos(A)\ge \rank (A)$ by relation (\ref{eqrelcp})).

\subsection{Numerical results for the completely positive rank}\label{sec:cpnumerical}

In this section, we explore the behaviour of the various bounds for the completely positive rank on three classes of examples. Our objective is to illustrate the superiority of  the ideal-sparse hierarchies compared  to the  dense ones. We examine both the quality of the bounds as well as computation times.

The first class we consider consists  of randomly generated sparse cp-matrices.
We will give the exact construction below.
In all numerical examples we considered for these matrices, 
the bounds obtained for $\xicpid_t(A)$ and $\xicpsp_t(A)$ were always at most $\rank(A)+2$. 
So we do not list the numerical bounds for these examples as there does not seem to be much  insight gained from them. 
However, random examples give us a way to compare the computation times amongst different hierarchies and across various matrix sizes, non-zero densities, and levels. 
In what follows the {\em non-zero density} of  $A\in \mathcal S^n$, denoted $\nzd(A)$,  is defined  as the proportion of non-zero entries above the main diagonal, i.e., $\nzd(A)=|E_A|/{n\choose 2}$.
Hence, a diagonal matrix has nzd=0, and a dense matrix has nzd=1. 

The second class contains  examples  from the literature, whose  cp-rank is known from theory.
However, recall  the moment hierarchies provide lower  bounds on $\taucp$, whose value is often unknown and could be strictly less than the  cp-rank. Regardless, these examples give an interesting testbed to evaluate the quality of the new bounds.

The third class of examples consists of  doubly nonnegative matrices, which are known to not be completely positive. In running these examples, the hope is to obtain an infeasibility certificate from the solver. This then numerically certifies that the matrix is not completely positive. In this context one hierarchy is said to perform better than another one if it returns the infeasibility certificate at a lower level or using less run time. 

The size of the matrices involved in the semidefinite programs  grows quickly with the level $t$ in the hierarchy (roughly, as $\binom{n+t}{t}$), so these problems become quickly too big for the solver (in particular, due to memory storage). 
We will consider matrices up to size $n=12$ for the dense and sparse hierarchies  at level $t=2$.
At level $t=3$ and for matrices of size $n=12$, we can only compute bounds for the weak sparse hierarchy.

All computations shown were run on a personal computer running  Windows 11 Home 64-bit with an 11th Gen Intel(R) Core(TM) i7-11800H @ 2.30GHz Processor and 16GB of RAM. The software we use was custom coded in Julia \cite{Julia} utilizing the JuMP \cite{JuMP} package for problem formulation, and MOSEK \cite{Mosek} as the semidefinite programming solver.\footnote{See the code repository {https://github.com/JAndriesJ/ju-cp-rank}{ju-cp-rank}
 for details.}

\begin{figure}[ht]
	\centering
	\text{
	 Computation times vs. matrix size and non-zero density, level $ t=2$}
	\includegraphics[scale=0.52]{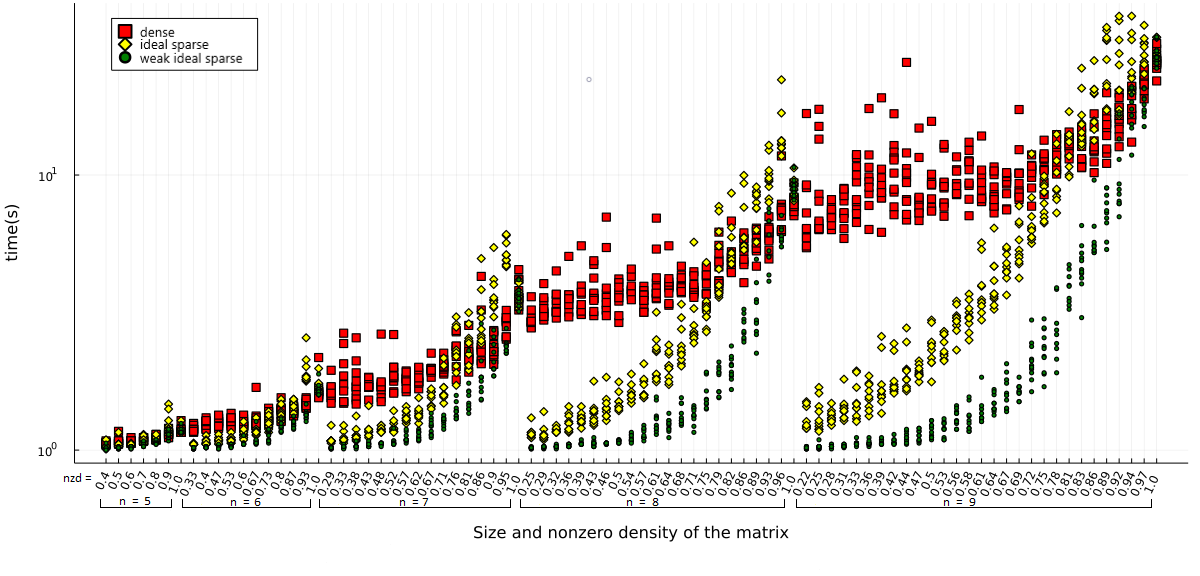}
	\caption{ Scatter plot of the computation times (in seconds) for the three hierarchies $\xicpid_{2,\dag}$ (indicated by a red square), $\xicpsp_{2,\dag}$ (indicated by a yellow losange), $\xicpwsp_{2,\dag}$ (indicated by a green circle) against matrix size and non-zero density for 850 random matrices, generated using the above described procedure. The matrices are arranged in ascending size ($n=5,6,7,8,9$) and then ascending non-zero density, ranging from the minimal density needed to have a connected support graph up to a fully dense matrix ($\nzd=1$).} 
	\label{scat1}
\end{figure}

\begin{figure}[ht]
	\centering
	\text{Computation times vs. matrix size and non-zero density, level $ t=3$}
        \includegraphics[scale=0.37]{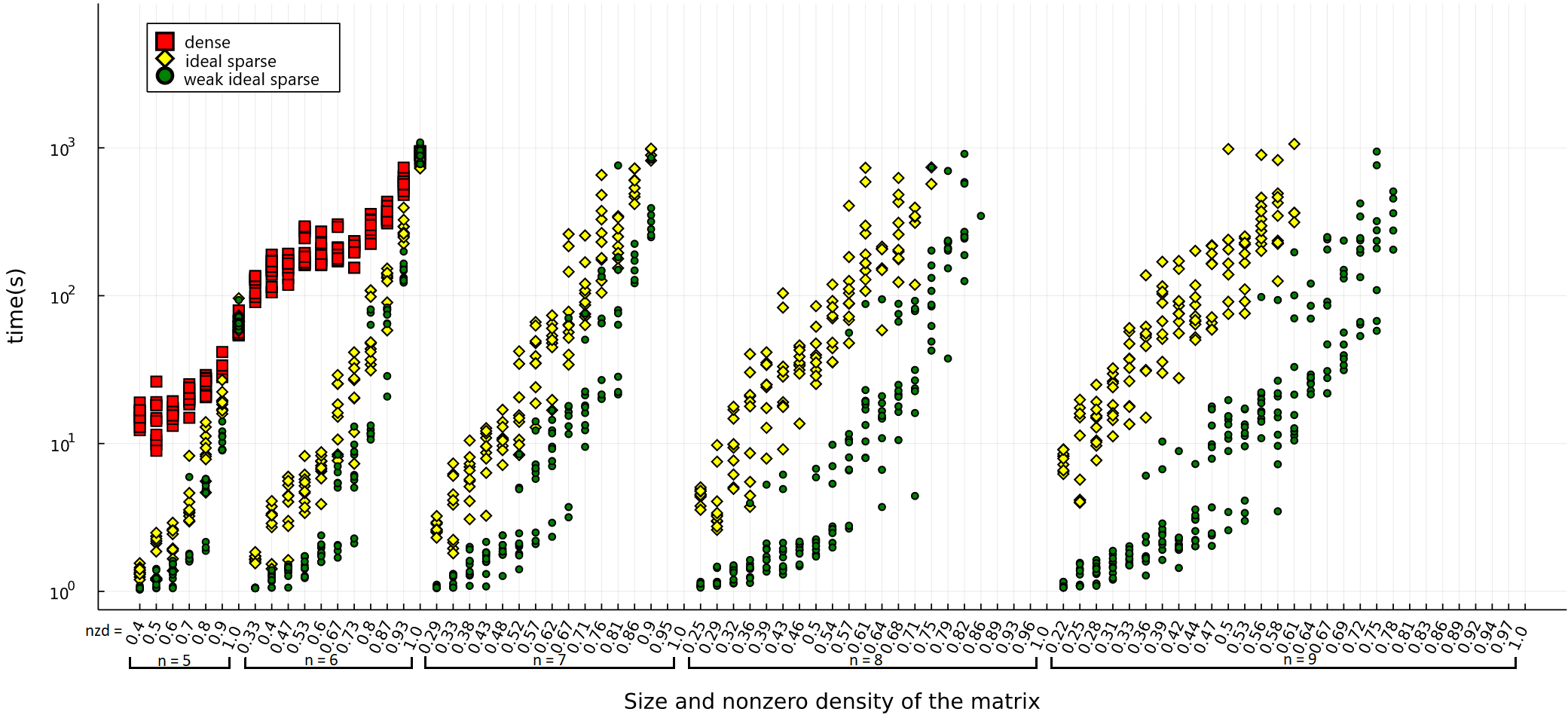}
	\caption{This is a similar plot to Figure \ref{scat1} but now for level t=3 of each of the hierarchies. By omitting markers we indicate that the corresponding computations either exceeded memory constraints or took longer than $10^3$ seconds.} \label{scat2}
\end{figure}

\subsubsection{Randomly generated sparse cp-matrices} \label{randgen}
We first describe how we construct random sparse cp-matrices.
Given  integers $n \in \N$ and  $n-1 \leq m \le \binom{n}{2}$, we create a symmetric $n\times n$ binary matrix $M$ with exactly $m$ ones above the diagonal, whose positions are selected uniformly at random. 
Let $G$ be the graph with $M$ as adjacency matrix.  
We only keep the instances where $G$ is a connected  graph.
We enumerate the maximal cliques $V_1,\ldots,V_p$ of $G$ (using, e.g., the  Bron-Kerbosch algorithm \cite{BK73}).
Then, we select a subset of maximal cliques $V_{q_1},...,V_{q_l}$   whose union covers every edge of $G$
(e.g., using a greedy algorithm). 
For each $k\in [l]$, generate $m_k\ge 1$  vectors $(a^{(k,i)})_{i \in [m_k]} \subseteq \R^n_+$ with uniformly random entries following ${\cal{U}}[0,1]$ and supported on  $V_{q_k}$. We will choose $m_k=2$ by default.
Then, consider the matrix  $\sum_{k\in[l]} \sum_{i \in [m_k]}  a^{(k,i)} (a^{(k,i)})^T$,  scale it so that all diagonal entries are equal to 1 and call $A$ the resulting matrix. By construction, $A$ is completely positive with connected support $G_A=G$, and non-zero density $\nzd= m/{n\choose 2}$.

We generate such random examples for varying matrix size ($n=5,6,7,8,9$) and incrementing the non-zero density $\nzd$ in ascending order. In order to not include examples with disconnected graphs we need $\nzd \geq (n-1)/{n\choose 2}$. To account for different graph configurations 
with the same non-zero density we generate $10$ examples per matrix size and nzd value. For all of them we compute the dense and (weak) sparse bounds of level $t=2$ and $t=3$. Here, we are not so much interested in the numerical bounds, but rather in their computation times.
This numerical experiment indeed permits to show the differences in computation time  between the ideal-sparse and dense hierarchies. It turns out that the computation times for the parameters  $\xicpid_t$,  $\xicpid_{t,\dag}$, and $\xicpid_{t,\ddagger}$ are all comparable at level $t=2,3$, likewise for the ideal-sparse analogs. For this reason, we only plot the  results for the ``$\dag$" variant, i.e., for the parameters $\xicpid_{t,\dag}$, $\xicpsp_{t,\dag}$, $\xicpwsp_{t,\dag}$.  The results are shown in Figure~\ref{scat1} (for $t=2$) and in Figure~\ref{scat2} (for $t=3$).

\medskip
We can make the following observations about the results in Figure \ref{scat1}.
As expected, the ideal-sparse hierarchy is faster to compute than the dense hierarchy for matrices with  non-zero density $\nzd \le 0.8$.
The computation of the weak ideal-sparse hierarchy is even faster. 
Moreover, the speed-up increases with the matrix size and the level of the hierarchy as can be seen across Figures \ref{scat1} and \ref{scat2}. At level $t=3$, some hierarchies can no longer be computed for certain matrix sizes and non-zero densities. This is particularly evident in the case of the dense hierarchy for matrices of size 7 and larger. The ideal-sparse hierarchies can be computed up to size 9 depending on the non-zero density. We show only the examples that we could compute in less than $10^3$ seconds. The parameters that either took longer that $10^3$ seconds or exceeded memory constraints can be inferred by the omission of their respective markers in Figure \ref{scat2}. 

\medskip
We also make an observation regarding how the values of the dense and weak-ideal sparse bounds compare for these random matrices. 
As observed earlier, 
the weak ideal-sparse hierarchy $\xicpwsp_t(A)$ is no longer guaranteed to be at least as strong as the dense hierarchy  $\xicp_t(A)$. 
Indeed, in our numerical experiments, we frequently observe the strict inequality $\xicpwsp_t(A) < \xicp_t(A)$ for randomly generated matrices $A$.
For example, the matrix (with entries rounded for presentation)
\begin{equation}\label{eqA}
A = \left( {\begin{array}{ccccc}
    1.0    &0.578  &0.0    &0.0    &0.225 \\
 	0.578  &1.0    &0.0    &0.0    &0.0    \\
 	0.0    &0.0    &1.0    &0.0    &0.656  \\
 	0.0    &0.0    &0.0    &1.0    &0.526  \\
 	0.225  &0.0    &0.656  &0.526  &1.0
  \end{array} } \right)
\end{equation}
has the following parameters at order $t= 2$:
$$
\Big(\xicpwsp_2(A)=4\Big)  < \Big(\xicp_2(A)=5\Big) \leq \Big(\xicpsp_2(A)=5\Big) \leq \Big(\cprank(A)=5\Big).
$$

\subsubsection{Selected sparse cp-matrices}\label{sec:sparsecp}

Here, we compute  the dense and (weak) ideal-sparse parameters 
for a few selected cp-matrices taken from the literature.
We first briefly  discuss the four example matrices we will consider, denoted ex1, ex2, ex3, ex4, and shown in relations (\ref{figex1}) and (\ref{figex2}) below.

\begin{equation}\label{figex1}
  \text{ex1} =
  \left( {\begin{array}{ccccc}
    3& 2&  0& 0& 1 \\
    2& 5&  6& 0& 0\\
    0& 6& 14& 4& 0\\
    0& 0&  4& 9& 1\\
    1& 0&  0& 1& 2
  \end{array} } \right),
\
  \text{ex3} =
  \left({\begin{array}{ccccccccccc}
781& 0& 72& 36& 228& 320& 240& 228& 36& 96& 0 \\
0& 845& 0& 96& 36& 228& 320& 320& 228& 36& 96\\
72& 0& 827& 0& 72& 36& 198& 320& 320& 198& 36\\
36& 96& 0& 845& 0& 96& 36& 228& 320& 320& 228\\
228& 36& 72& 0& 781& 0& 96& 36& 228& 240& 320\\
320& 228& 36& 96& 0& 845& 0& 96& 36& 228& 320\\
240& 320& 198& 36& 96& 0& 745& 0& 96& 36& 228\\
228& 320& 320& 228& 36& 96& 0& 845& 0& 96& 36\\
36& 228& 320& 320& 228& 36& 96& 0& 845& 0& 96\\
96& 36& 198& 320& 240& 228& 36& 96& 0& 745& 0\\
0& 96& 36& 228& 320& 320& 228& 36& 96& 0& 845
 \end{array} } \right),
\end{equation}

\begin{equation}\label{figex2}
  \text{ex2} =
  \left( {\begin{array}{ccccc}
   2& 0& 0& 1& 1 \\
   0& 2& 0& 1& 1 \\
   0& 0& 2& 1& 1 \\
   1& 1& 1& 3& 0 \\
   1& 1& 1& 0& 3 \\
  \end{array} } \right),
\
\text{ex4} =
  \left( {\begin{array}{cccccccccccc}  
91&  0&  0&  0& 19& 24& 24& 24& 19& 24& 24& 24\\
0& 42&  0&  0& 24&  6&  6&  6& 24&  6&  6&  6\\
0&  0& 42&  0& 24&  6&  6&  6& 24&  6&  6&  6\\
0&  0&  0& 42& 24&  6&  6&  6& 24&  6&  6&  6\\
19& 24& 24& 24& 91&  0&  0&  0& 19& 24& 24& 24\\
24&  6&  6&  6&  0& 42&  0&  0& 24&  6&  6&  6\\
24&  6&  6&  6&  0&  0& 42&  0& 24&  6&  6&  6\\
24&  6&  6&  6&  0&  0&  0& 42& 24&  6&  6&  6\\
19& 24& 24& 24& 19& 24& 24& 24& 91&  0&  0&  0\\
24&  6&  6&  6& 24&  6&  6&  6&  0& 42&  0&  0\\
24&  6&  6&  6& 24&  6&  6&  6&  0&  0& 42&  0\\
24&  6&  6&  6& 24&  6&  6&  6&  0&  0&  0& 42  
   \end{array} } \right).
\end{equation}

The matrix ex1 (from \cite{98B}) is supported on the 5-cycle $C_5$ and
the matrix ex2 (from  \cite{98XX}) is  supported on the bipartite graph $K_{3,2}$.
In both cases, 
we have $\xicpsp_1(A) = \cprank(A) = |E_A|$ (combining Lemma \ref{lemcplbecc}  and the results of \cite{DJL94} mentioned earlier at the end of Section \ref{seccp2}).
The matrices ex3 and ex4 were constructed, respectively, in 
\cite{BSU2014,BSU2015} as examples of matrices having a large cp-rank
exceeding the value $n^2/4$ (thus refuting the conjecture by Drew et al. \cite{DJL94}).
The matrix ex3 is supported on $\overline {C_{11}}$, the complement of an 11-cycle, and matrix ex4 is supported on the complete tripartite graph $K_{4,4,4}$.
One can verify that the edge clique-cover number  is equal to 8 for $\overline {C_{11}}$ and  to 16 for $K_{4,4,4}$.

\medskip
The numerical results for these four examples are presented in Table \ref{tab:table1}, where we also show  other parameters for the matrix (size $n$, rank $r$, cp-rank $r_{\text{\rm cp}}$) and its support graph  (number $p$ of maximal cliques, edge clique-cover number $c$).
 Here are some comments about  Table  \ref{tab:table1}.
 
 The results confirm the results in 
 Lemma~\ref{lemcplbecc}: the ideal-sparse bound of level $t=1$ is equal to the number of edges for ex1 and ex2 (and matches the cp-rank); moreover it gives a strong improvement on the dense bound of level 1. The bounds of level $t=2$ all exceed the rank of the matrix (as expected in view of (\ref{eqrelcp})). 
 At level $t=3$, only the weak ideal-sparse bound can be computed for the matrices ex3 and ex4.

\begin{table}[ht]
\caption{Dense and {ideal-}sparse bounds for selected sparse cp-matrices}
  \begin{center}
    \label{tab:table1}
    \begin{tabular}{cl|l|l|l||c|c|c|c||r|r|r|r} 
    \text{$A$}&$n$ &$p$ &$c$ &$r$ &\multicolumn{3}{c|}{bounds}&$r_{\mathrm{cp}}$&   \multicolumn{3}{c}{times (seconds)}
     \\
      &&&&&$\xicpid_1$&$\xicpsp_1$&$\xicpwsp_1$&& {dense} &{ideal-sparse}&wisp \\ \hline
ex1&	5&	5&	5&5	&2.71&5&5&5			&$<1$&$<1$&$<1$\\
ex2 &	5&	6&	6&4	&3&6&6&6&$<1$	 		&$<1$&$<1$\\
ex3&	11&	22&	$8$&11	&4.24&8.53&8.53&32&$<1$&$<1$&$<1$\\
ex4 &	12&	64&	16&10	&4.85&29.66&29.63&37&$<1$&$<1$&$<1$\\\hline
&&&&&$\xicpid_{2,\ddagger}$&$\xicpsp_{2,\ddagger}$&$\xicpwsp_{2,\ddagger}$&&&& \\ \hline
ex1&5&&&5&5&5&5&5&$<1$&$<1$&$<1$\\
ex2&5&&&4&6&6&6&6&$<1$&$<1$&$<1$\\
ex3&11&&&11&21.93&22.32&22.32&32&123.86&54.89&8.14\\
ex4&12&&&10&29.57&29.66&29.66&37&238.94&33.78&1.28\\ \hline
&&&&&$\xicpid_{3,\ddagger}$&$\xicpsp_{3,\ddagger}$&$\xicpwsp_{3,\ddagger}$&&&& \\ \hline
ex3&11&&&11&-&-&22.33&32&-&-&2648.69\\
ex4&12&&&10&-&-&29.66&37&-&-&28.69\\
\multicolumn{12}{@{}p{5.5in}}{\footnotesize     }\\
\multicolumn{12}{@{}p{5.5in}}{\footnotesize $n=$ size of $A$, $p=$  number of maximal cliques of $G_A$,  $c=$ edge clique-cover number of $G_A$,  $r=\rank (A)$,  $r_{\mathrm{cp}}= \cprank(A)$ }\\
 \multicolumn{12}{@{}p{5.5in}}{\footnotesize - $:$ computations failed due to memory constraints}\\
    \end{tabular}
  \end{center}
\end{table}

In Table \ref{tab:table1},  the values of the bounds at level $t=3$ are close to those at level $t=2$ for matrices ex3 and ex4. 
However,  the tests for the flatness condition (\ref{eqflatk}) fail, so that one cannot  claim that the bounds are equal to $\taucp$ at this stage. 

\medskip
We also tested whether the flatness conditions (\ref{eqflat}) and (\ref{eqflatk}) hold for matrices ex1 and ex2 at level $t=2$, and whether one can extract atoms and construct a cp-factorization. 

The results are summarized in Table \ref{tab:table2}, where we indicate the number of atoms (corresponding to a cp-factorization with that many factors) when the extraction procedure is successful. We indicate that the extraction procedure fails by reporting ``$\#$ atoms=0''. 
 As mentioned in \cite{HL05}, one may indeed try and apply the extraction procedure even if flatness does not hold.
 
For the  dense bounds of level $t=2$,  flatness does not hold for the matrices ex1 and ex2. However, while  one does not succeed to extract atoms for matrix ex1, the extraction is successful for matrix ex2 and returns 6 atoms. 
Interestingly, flatness holds for the ideal-sparse bounds and the atom extraction is successful.  However, the number of extracted atoms is 10 for matrix ex1, thus twice the cp-rank. To verify that the extracted atoms are (approximatively)  correct,  we use them to construct a cp-matrix $A_{\rm{rec}}$, which we then compare to the original matrix $A$.  In all cases we obtain  $\|A_{\rm{rec}} -A \|_1 \leq 10^{-8}$, which shows that a correct factorization has been constructed.

Note that for the ideal-sparse parameter, since one splits the problem over the maximal cliques and has a distinct linear functional $L_k$ for each clique $V_k$, it may be more difficult to satisfy the flatness condition (\ref{eqflatk}) (since each $L_k$ must satisfy it), as happens for matrices ex3 and ex4. 
\begin{table}[ht]
 \caption{Testing flatness and atom extraction}
  \begin{center}
    \label{tab:table2}
    \begin{tabular}{c|c|c||c|c||c|c} 
   $A$ &\multicolumn{2}{|c||}{ $\xicpid_{2,\ddagger}$} & \multicolumn{2}{|c||}{ $\xicpsp_{2,\ddagger}$} & \multicolumn{2}{|c}{ $\xicpwsp_{2,\ddagger}$} \\ \hline
    & flatness (\ref{eqflat}) & \# atoms & flatness  (\ref{eqflatk}) & \# atoms & flatness (\ref{eqflatk}) & \# atoms \\
ex1 &	false&	0&	true& 10&	false&	0\\
ex2 &	false&	6&	true&  6&   true&	6 \\
\hline
  &\multicolumn{2}{|c||}{ $\xicpid_{3,\ddagger}$} & \multicolumn{2}{|c||}{ $\xicpsp_{3,\ddagger}$} & \multicolumn{2}{|c}{ $\xicpwsp_{3,\ddagger}$} \\ \hline
ex1 &	false&	10&	true& 10 &	false&	0\\
ex2 &	true&	6&	true&  6&   true&	6 \\
    \end{tabular}
  \end{center}
\end{table}

\subsubsection{Doubly nonnegative matrices that are not completely positive}

In this section we consider the following three matrices that are known to be doubly nonnegative but not completely positive (taken from  \cite{93SZ,Nie14,98B}):
\begin{equation*}
  \text{ex5} =
  \left( {\begin{array}{ccccc}
 1& 1& 0& 0& 1\\
 1& 2& 1& 0& 0\\
 0& 1& 2& 1& 0\\
 0& 0& 1& 2& 1\\
 1& 0& 0& 1& 3
  \end{array} } \right),
 ~~
  \text{ex6} =
  \left( {\begin{array}{ccccc}
  1& 1& 0& 0& 1\\
  1& 2& 1& 0& 0\\
  0& 1& 2& 1& 0\\
  0& 0& 1& 2& 1\\
  1& 0& 0& 1& 6
   \end{array} } \right),
  ~~
  \text{ex7} =  
   \left( {\begin{array}{cccccc}
   7&  1&  2&  2&  1&  1\\
   1& 12&  1&  3&  3&  5\\
   2&  1&  2&  3&  0&  0\\
   2&  3&  3&  5&  0&  0\\
   1&  3&  0&  0&  2&  4\\
   1&  5&  0&  0&  4& 10\\
    \end{array} } \right).
\end{equation*}

The objective is to see whether the hierarchies are able to detect that the matrix is not cp. This can be achieved in two ways: when the solver  returns an infeasibility certificate,  or when it returns a bound that exceeds a known upper bound on the cp-rank. 
 We test this for the bounds at level $t=1$ and $t=2$. At level $t=2$ we try different variants by adding the constraints (\ref{cpdagger1}), (\ref{cpdagger2}), (\ref{cpdagger3}), and (\ref{eqcpxx}) and their sparse analogs. The results are presented in Tables \ref{tab:table3} and \ref{tab:table4}. 
 
 There we indicate one of three possible outcomes.
  The first outcome is indicated with a question mark ``?", which indicates that the solver could not reach a decision within the default MOSEK solver parameters.  The second possible outcome is when the solver returns an infeasibility certificate (indicated with *), or when it returns a value that exceeds a known  upper bound for the cp-rank (in which case the bound is marked with *). The last column in both tables, labeled ``$r_{\text{\rm cp}}\le $", provides such an upper bound on the cp-rank of a cp-matrix with the given support graph.
The third possible outcome is when the solver returns a value that does not violate the upper bound, in which case no conclusion can be reached.
All computations took less than a second and hence times are not shown.

\begin{table}[h!]
  \begin{center}
    \caption{Detecting non-cp matrices for $t=1$.}
    \label{tab:table3}
    \begin{tabular}{l|l|l|l|c|c|c} 
  $A$& $n$ &  $r$ &$\xicpid_1$ &$\xicpsp_1$ &$\xicpwsp_1$  & $r_{\rm{cp}} \leq$\\
      \hline
ex5&5&4&2.47&*&*&5\\
ex6&5&5&2.59&*&*&5\\
ex7&6&6&2.4&3.02&3.02&17\\
    \multicolumn{3}{@{}p{1.5in}}{\footnotesize * $=$ infeasibility certificate}\\
   \end{tabular}
  \end{center}
\end{table}

\begin{table}[h!]
\caption{Detecting non cp-matrices for $t=2,3$.}
  \begin{center}
    \label{tab:table4}
    \begin{tabular}{l|l|l |c c c |c c c| c c c|c} 
     $A$ & $n$ & $r$ &$\xicpid_2$ &$\xicpsp_2$ &$\xicpwsp_2$ & $\xicpid_{2,\dag}$ &
     $\xicpsp_{2,\dag}$& $\xicpwsp_{2,\dag}$  & $\xicpid_{2,\ddagger}$
     &$\xicpsp_{2,\ddagger}$ &$\xicpwsp_{2,\ddagger}$ & $r_{\rm{cp}} \leq$\\
      \hline
ex5&5&4&?&*&?&?&*&*&*&*&*&5\\
ex6&5&5&13.56*&?&*&13.56*&*&*&16.11*&*&*&5\\
ex7&6&6&?&34.88*&34.01*&12.94&*&*&13.89&*&*&17\\
\hline
      &  &  &$\xicpid_3$ &$\xicpsp_3$ &$\xicpwsp_3$ & $\xicpid_{3,\dag}$ &
     $\xicpsp_{3,\dag}$& $\xicpwsp_{3,\dag}$  & $\xicpid_{3,\ddagger}$
     &$\xicpsp_{3,\ddagger}$ &$\xicpwsp_{3,\ddagger}$ & \\
      \hline
ex5&5&4&?&	*&	?&	*&	*&	*&	*&	*&	*&5\\
ex6&5&5&?&	*&	*&	?&	*&	*&	194.19*&	*&	*&5\\
ex7&6&6&?&	?&	?&	?&	*&	*&	?&	*&	*&17\\
 \multicolumn{12}{@{}p{5.5in}}{\footnotesize * $=$ infeasibility certificate, ~ ? $=$ unknown result status}
    \end{tabular}
  \end{center}
\end{table}

We make three observations about Tables \ref{tab:table3}-\ref{tab:table4}. The first is that the ideal-sparse hierarchies show infeasibility at level $t=1$ already for examples ex5 and ex6 while the dense hierarchy shows the same only at level $t=2$ with all additional constraints imposed. Secondly, the ideal-sparse hierarchy correctly identifies ex7 as not cp at level $t=2$ while the dense hierarchy does not succeed even at level $t=3$. The third observation is that adding additional constraints helps prevent the solver from returning an ``unknown result status" but this seems to be less needed in the case of the ideal-sparse hierarchies.
It should be noted that increasing the level of the hierarchy creates more opportunity for numerical errors in the computations, as seen in Table~\ref{tab:table4}.

\section{Application to the nonnegative rank}\label{sec:app-nn}

In this section we indicate how the treatment in the previous section for the cp-rank extends naturally to the asymmetric setting of the nonnegative rank.

\subsection{Ideal-sparsity bounds for the nonnegative rank}

Given a nonnegative matrix $M\in \R^{m\times n} $, its {\em nonnegative rank}, denoted $\nnrank(M)$, is  the smallest integer $r$ for which there exist nonnegative vectors $a_\ell \in\R^m_+$ and $b_\ell\in \R^n_+$ such that 
\begin{equation}\label{eqnnM}
M=\sum_{\ell=1}^r a_\ell b_\ell^T.
\end{equation}
Computing the nonnegative rank is an NP-hard problem \cite{Vav09}.
Fawzi and Parrilo \cite{FP16} introduced  the following natural ``convexification" of the nonnegative rank:
$$
\tau_+(M)=\inf\Big\{\lambda: {1\over \lambda} M \in \text{conv}\{ xy^T: x\in \R^m_+, y\in \R^n_+, M \ge xy^T\}\Big\},$$
which can be seen as an asymmetric analog of $\taucp$. 
We consider  the analogs of the parameters $\xicpid_t$ and $\xicpsp_t$, which now involve  linear functionals acting on polynomials in $m+n$ variables. As in the introduction,  set $V=[m+n]=U\cup W$, where $U=[m]=\{1,\ldots,m\}$ (corresponding to the row indices of $M$) and $W=\{m+1,\ldots, m+n\}$ (corresponding to the column indices of $M$, up to a shift by $m$).
Set 
$$E^M=\{\{i,j\}\in U\times W: M_{i,j-m}\ne 0\},$$ 
so that  the bipartite graph  $G^M=(V=U\cup W, E^M)$  corresponds to the support graph of $M$. 
We also set $\olE^M=(U\times W)\setminus E^M$ and
$M_{\max}=\max_{i\in U,j\in W}M_{i,j-m}$.
As is  well-known (see, e.g., \cite{GdLL2019a}),   the vectors in (\ref{eqnnM}) may be assumed to satisfy 
$\|a_\ell\|_\infty, \|b_\ell\|_\infty \le \sqrt{M_{\max}}$ (after rescaling). This motivates the definition of the semialgebraic set $K^M$ from (\ref{eqsetKM}) and,
for any integer $t \ge 1$,  of the  parameter: 
  \begin{align}\xinnid_t(M) = \min\{L(1): & \ L\in \R[x_1,\ldots,x_{m+n}]^*_{2t},\\
&   L(x_ix_j)=M_{i,j-m}\ (i\in U, j\in W), \label{eqnnA} \\
& L([x]_t[x]_t^T)\succeq 0, \label{eqnn0}\\
&L((\sqrt{M_{\max}}x_i-x_i^2)[x]_{t-1}[x]_{t-1}^T)\succeq 0 \ \text{ for } i\in V, \label{eqnnxi}\\
&L((M_{i,j-m}-x_ix_j)[x]_{t-1}[x]_{t-1}^T)\succeq 0 \ \text{ for } \{i,j\}\in E^M,\label{eqnnxixj}\\
& L(x_ix_j[x]_{2t-2})=0 \ \text{ for } \{i,j\}\in \olE^M\}. \label{eqnnid}
\end{align}
If we omit the (ideal) constraint (\ref{eqnnid}) and require the constraint (\ref{eqnnxixj}) to hold also for pairs $\{i,j\}\in \olE^M$, then we obtain the (weaker)  parameter $\xinn_t(M)$,  introduced in \cite{GdLL2019a} as a lower bound on $\tau_+(M)$ (and thus on $\rank_+(M)$).

In addition, we can   define  ideal-sparse bounds, by further exploiting the sparsity pattern of $M$. 
As the support graph $G^M$  is now a bipartite graph it is convenient to use the following notion of biclique. A {\em biclique} in $G^M$ corresponds to a complete bipartite subgraph and it is thus given by a pair $(A,B)$ with $A\subseteq U$ and $B\subseteq W$ such that  $\{i,j\}\in E^M$ for all $(i,j)\in A\times B$; it is maximal if $A\cup B$ is maximal.
Let $V_1=A_1\cup B_1,\ldots, V_p=A_p\cup B_p$ be the vertex sets of  the maximal bicliques in $G^M$
and, for any integer $t\ge 1$,  define the parameter
  \begin{align}\xinnsp_t(M)= \min\Big\{& \sum_{k=1}^pL_k(1):  \ L_k\in \R[x(V_k)]^*_{2t}\ (k\in [p]),\\
&   \sum_{k\in [p]: \{ i,j\}\subseteq  V_k}  L_k(x_ix_j)=M_{i,j-m}\ (i\in U, j\in W), \label{eqnnAsp} \\
& L_k([x(V_k)]_t[x(V_k)]_t^T)\succeq 0 \ (k\in [p]), \label{eqnn0sp}\\
&L_k((\sqrt{M_{\max}}x_i-x_i^2)[x(V_k)]_{t-1}[x(V_k)]_{t-1}^T)\succeq 0 \ (i\in V_k,\ k\in [p]), \label{eqnnxisp}\\
&L_k((M_{i,j-m}-x_ix_j)[x(V_k)]_{t-1}[x(V_k)]_{t-1}^T)\succeq 0 \ ( i\in U, j\in W, \{i,j\}\subseteq V_k,\ k\in [p])\Big\}.\label{eqnnxixjsp}
\end{align}
Summarizing, we have the following inequalities among the above parameters
$$
\xinn_{t-1}(M)\le \xinnid_t(M)\le \xinnsp_t(M)\le\tau_+(M)\le  \nnrank (M) \ \text{ for any } t\ge 2,
$$
with asymptotic convergence of all bounds  to $\tau_+(M)$; this was shown in  \cite{GdLL2019a} for the bounds $\xinn_t(M)$ (and this also follows as an application of Theorem \ref{theoconvGMP}).

As in the case of the cp-rank, there are more constraints that may be added to the above programs to strengthen the bounds.
In \cite{GdLL2019a} the authors propose to exploit the nonnegativity of the variables and add the constraints
\begin{align}
L([x]_{2t})\ge 0,\label{nndag1}\\
L((\sqrt{M_{\max}}x_i-x_i^2) [x]_{2t-2})\ge 0 \text{ for } i\in V,\label{nndagxi}\\
L((M_{i,j-m}-x_ix_j)[x]_{2t-2})\ge 0\text{ for } (i,j)\in U\times W.\label{nndagxixj}
\end{align} 
Let $\xinn_{t,\dag}(M)$ denote the parameter obtained by adding the constraint (\ref{nndagxixj})
to $\xinn_t(M)$.
Similarly, one may add (\ref{nndagxixj})  to the parameter $\xinnid_t(M)$ (requiring (\ref{nndagxixj}) only for pairs in $E^M$) and its sparse analog to $\xinnsp_t(M)$, leading, respectively, to the parameters
$\xinnid_{t,\dag}(M)$ and $\xinnsp_{t,\dag}(M)$.
So, $\xinn_{t,\dag}(M)\le \xinnid_{t,\dag}(M)\le \xinnsp_{t,\dag}(M)$.
Finally, we also introduce the parameters, where we use the symbol $\ddagger$ instead of $\dag$ when adding all the constraints 
(\ref{nndag1}), (\ref{nndagxi}), (\ref{nndagxixj}).

\subsection{Links to combinatorial lower bounds  on the nonnegative rank}
We now recall some other known lower bounds on the nonnegative rank and indicate their relations to the parameters considered here.

Fawzi and Parrilo \cite{FP16} introduced a semidefinite bound $\taunnsos(M)$ and show it satisfies
$\taunnsos(M)\le \taunn(M)$. In \cite{GdLL2019a} it is shown that the parameters $\xinn_{2,\dag}(M)$ strengthen this bound\footnote{This follows from the proof of  \cite[Proposition 15]{GdLL2019a}, since it only uses the relation $L((M_{ij}-x_ix_j)x_ix_j)\ge 0$ for any $(i,j)\in U\times W$ in addition to the constraints defining the basic parameter $\xinnid_2(M)$.}:
$$\taunnsos(M)\le \xinn_{2,\dag}(M)\le \taunn(M).$$

There is a well-known combinatorial lower bound on the nonnegative rank, which can be seen as an asymmetric analog of the lower bound on  the cp-rank of $A$ given by the edge clique-cover number $\text{\rm c}(G_A)$.
Recall $G^M=(U\cup W, E^M)$ is the bipartite graph defined as the support graph of $M\in \R^{m\times n}_+$. Define the {\em edge biclique-cover number} of  $G^M$, denoted $\bc(G^M)$,  as the smallest number of bicliques whose union covers every edge in $E^M$. Then, we have
$$\nnrank(M)\ge \bc(G^M).$$ 
As a biclique in $G^M$ corresponds to a pair $(A,B)\subseteq U\times W$ for which the rectangle $A\times B$ is fully contained in the support of $M$, the parameter $\bc(G^M)$ is also known as the {\em rectangle covering number} of $M$ (see, e.g., \cite{FP16,GG2012}).
Define its fractional analog $\bc_{\text{\rm frac}}(G^M)$ as
\begin{equation}\label{eqbcfrac}
\bc_{\text{\rm frac}}(G^M)=\min\Big\{\sum_{k=1}^p x_k: x\in \R^k_+,\ \sum_{k:\{i,j\}\subseteq V_k} x_k\ge 1\ \text{ for } \{i,j\}\in E^M\Big\} \le \bc(G^M).
\end{equation}

Yet another well-known combinatorial interpretation of bicliques is as follows.
Define the rectangular graph $\RG(M)$, with vertex set $E^M$ and where two distinct pairs $\{i,j\},\{k,\ell\}\in E^M$ form an edge of $\RG(M)$ if $M_{i\ell}M_{kj}=0$. In other words, $\{i,j\},\{k,\ell\}\in E^M$ do not form an edge in $\RG(M)$ precisely if $(\{i,k\}, \{j,\ell\})$  corresponds to a biclique in $G^M$. Then, the parameter $\bc(G^M)$ coincides with the coloring number of $\RG(M)$ and $\bc_{\text{\rm frac}}(G^M)$ coincides with $\chi_f(\RG(M))$, the fractional coloring number of $\RG(M)$. So, $$ \nnrank(M)\ge \bc(G^M)=\chi(\RG(M)).$$
The following relationships are shown in \cite{FP16}:
$$\taunn(M)\ge \chi_f(\RG(M))=\bc_{\text{\rm frac}}(G^M),\ \ 
\taunnsos(M)\ge \overline \vartheta(\RG(M)),$$
where  $\overline \vartheta(\RG(M))$ is the theta number of the complement of $\RG(M)$.
As we now observe, the ideal-sparse parameter $\xinnsp_1(M)$  is at least as good as $\bc_{\text{\rm frac}}(G^M)$, which is the analog of Lemma \ref{lemcplbecc}.

\begin{lemma}\label{lemnnlb}
For $M\in \R^{m\times n}_+$  we have $ \xinnsp_1(M)\ge\bc_{\text{\rm frac}}(G^M)$.
\end{lemma}
\begin{proof}
Let $(L_1,\ldots, L_p)$ be an optimal solution for $\xinnsp_1(M)$. Then, 
$L_k(M_{i,j-m}-x_ix_j)\ge 0$ for each $k\in [p]$ and  $\{i,j\}\in E^M$  such that $\{i,j\}\subseteq V_k$. 
As $\sum_{k: \{i,j\}\subseteq V_k}L_k(x_ix_j)=M_{i,j-m}$,  this implies
$\sum_{k: \{i,j\}\subseteq V_k} L_k(1) \ge 1$ for each $\{i,j\}\in E^M$. Hence, the vector $x=(L_k(1))_{k=1}^p$ provides a feasible solution to program (\ref{eqbcfrac}), which implies $\sum_{k=1}^pL_k(1)\ge \bc_{\text{\rm frac}}(G^M)$.
\end{proof}

As for the cp-rank, we now give a class of matrices showing a large separation between the ideal-sparse and dense bounds of level $t=1$.

\begin{example} \label{ex:xi_1_sep}
Consider the identity matrix $M=I_n\in \mathcal S^n$. Clearly, we have $\cprank(I_n)=\rank(I_n)=n$. As the support graph $G^M$ is the disjoint union of $n$ edges, its fractional edge biclique-cover number is equal to $n$ and thus, in view of Lemma \ref{lemnnlb}, we have $\xinnsp_1(I_n)=n=\nnrank(I_n)$.
We now show that for the dense bound, we have $\xinnid_1(I_n)< 8$ for any $n\ge 4$.
For this recall that $\xinnid_1(I_n)$ is given by 
$$\xinnid_1(I_n)=\min\{L(1): L\in \R[x]_2^*,\ L(x_i)\ge L(x_i^2) \ (i\in [2n]), \ L(x_ix_{n+j})=\delta_{i,j} \ (i,j\in [n]),\ L([x]_1[x]_1^T)\succeq 0\}.$$
Consider the linear functional $L\in \R[x]_2^*$ defined by
$L(1) = 8{n-2\over n}$, $L(x_i)=L(x_i^2)= 2{n-2\over n}$ for $i\in [2n]$, $L(x_ix_j)=L(x_{n+i}x_{n+j})= {n-4\over n}$ for $i\ne j\in [n]$, and $L(x_ix_{n+j})=\delta_{i,j}$ for $i,j\in [n]$. Then one can check that
$$L([x]_1[x]_1^T) = \left(\begin{matrix}
8{n-2\over n} & 2{n-2\over n}e^T & 2{n-2\over n}e^T \cr
2{n-2\over n}e & I_n+ {n-4\over n}J_n & I_n\cr
2{n-2\over n}e & I_n & I_n+ {n-4\over n}J_n
\end{matrix}\right)\succeq 0.
$$
Hence, $L$ is feasible for the program defining $\xinnid_1(I_n)$, which shows $\xinnid_1(I_n)\le L(1)=  8{n-2\over n}<8.$
\end{example}

\subsection{Numerical results for the nonnegative rank}\label{sec:nnnumerical}
In this section we test the ideal-sparse and dense hierarchies on two classes of nonnegative matrices. The first class consists of size $4 \times 4$ matrices that depend continuously on a single variable. The second class we consider are the Euclidean distance matrices (EDMs).

\subsubsection{Matrices related to the nested rectangles problem}
The nonnegative matrices we will consider  have an interesting link between their nonnegative rank and the geometric nested rectangles problem (see \cite{BFPS15}).  Bounds for their nonnegative rank were investigated by Fawzi and Parrilo \cite{FP16} and Gribling et al.  \cite{GdLL2019a}.
Consider the  matrices
$$
S(a,b) := 
  \left(\begin{array}{@{}cccc@{}}
    1-a & 1+a & 1-b & 1+b  \\
    1+a & 1-a & 1-b & 1+b \\
    1+a & 1-a & 1+b & 1-b  \\
    1-a & 1+a & 1+b & 1-b 
  \end{array}\right) \quad \text{ for } a,b\in [0,1].
$$
If $a,b<1$, then $S(a,b)$ is fully dense and no improvement can be expected from our new bounds. 
Thus, we consider the case $b=1$ and  $a\in [0,1]$.  We have computed the bounds  $\xinnid_{t,\ddagger}(M)$  and $\xinnsp_{t,\ddagger}(M)$ at level $t =1,2,3$ for $M = S(a,1)$ with $a$ ranging from $0$ to $1$ in increments of $0.01$. The results are displayed in Figure \ref{longplot} below. 
We can make the following two observations about  Figure \ref{longplot}. First,  the ideal-sparse hierarchy is much stronger at level $t=1$,  but at level $t=2$ the dense and ideal-sparse hierarchies give comparable bounds. Second, for $a=1$,  all bounds, except the dense bound of level 1, are equal to $4=\nnrank(S(1,1))$ (as is  expected for the ideal-sparse hierarchy in view of Lemma~\ref{lemnnlb}).

\begin{figure}[ht]
	\centering
	\textbf{Bounds $\xinnid_{t,\ddag}(M)$ and $\xinnsp_{t,\ddag}(M)$  for $M=S(a,1)$ and $t=1,2,3$ versus $a\in[0,1]$ }\par\medskip
	\includegraphics[scale=0.60]{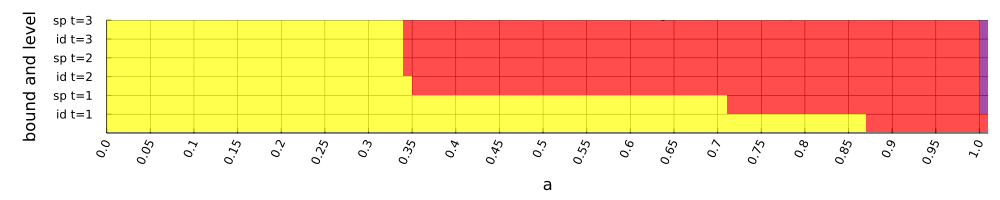}
	\caption{This figure shows $\xinnid_{t,\dag}(S(a,1))$ and $\xinnsp_{t,\dag}(S(a,1))$ computed at levels $t =1,2,3$  with $a$ ranging from $0$ to $1$ in increments of $0.01$. The colour indicates a lower bound on the obtained numerical value: yellow, red and purple show the bound is at least 2, 3, and 4, respectively. So a red square at $a=0.35$ and ``sp t=2" means  $\xinnsp_{2,\dag}(M) \geq 3$. }
	\label{longplot}
\end{figure}

\subsubsection{Euclidean distance matrices}\label{secEDM}
The second class of examples we consider 
are the Euclidean distance matrices $M_n=((i-i)^2)_{i,j=1}^n\in\R^{n\times n}_+$, known to have a large separation between their rank and their nonnegative rank.
Indeed, $\rank(M_n)=3$ \cite{BL2009}, and their bipartite support graph $G^{M_n}$ is  $K_{n,n}$ with a deleted perfect matching (known as a {\em crown graph}), whose edge biclique-cover number satisfies  $ \bc(G^{M_n})=\Theta(\log n)$
 \cite{CGP}. So we have $\rank({M_n})=3$ and  $\nnrank (M)\ge \bc(G^{M_n})=\Theta(\log n)$. In addition, it is known that $\nnrank({M_n})\le 2+ \lceil \frac{n}{2} \rceil $, see \cite[Theorem~9]{GG2012}.
The numerical  results are shown in Table \ref{tab:table5}. 
In these examples, the ideal-sparse bound of level $t=2$ is more difficult to compute, since the support graph $G^M$ has  $2^{n-1}$ maximal bicliques, each with $n$ vertices. For this reason we could compute $\xinnsp_{2,\dag}$ only until $n=7$ before running out of memory.  
So this example illustrates the limitations of the ideal-sparsity approach, when the number of maximal cliques is too large. Note  that this difficulty -- large number of maximal bicliques --  remains even if we would replace the support graph $G^{M_n}$ by a supergraph $\tilde G$, obtained by adding to $M^{G_n}$ (say) $s$ edges from the missing perfect matching. Indeed, such $\tilde G$ still has $2^{n-s-1}$ maximal bicliques, each with $n+s$ vertices.

\begin{table}[h!]
\caption{Bounds for the matrices $M=((i-j)^2)_{i,j=1}^n$.}
  \begin{center}
    \label{tab:table5}
\begin{tabular}{ c | c || c | c || c | c || c } 
$n$ & bc & $\xinnid_{1,\dag}$ & $\xinnid_{2,\dag}$ & $\xinnsp_{1,\dag}$ & $\xinnsp_{2,\dag}$ & 
$2+ \lceil \frac{n}{2} \rceil $ \\
      \hline
4& 4 &2&3.46&3&3.63&4 \\
5& 4&2&3.73&3.35&4.19&5 \\
6& 4 &2&3.96&3.41&4.53&5 \\
7& 5 &2&4.17&3.55&4.85&6 \\
8& 5&2&4.35&3.59&-&6 \\
9& 5&2&4.51&3.66&-&7 \\
   \end{tabular}
  \end{center}
\end{table}

\section{{Concluding remarks}}\label{sec:final}

In this paper we have introduced a new sparsity approach for GMP, which arises when in the formulation of GMP one has explicit ideal-type constraints that require the support of the measure to be contained in the variety of an ideal generated by monomials $x_ix_j$ corresponding to (the non-edges of) a graph $G$. We compared it to the more classic correlative sparsity structure that requires a chordal structure on the graph $G$, while our new ideal-sparse hierarchy does not need it. We explored its application to  the problem of bounding the nonnnegative rank and the cp-rank of a matrix and illustrate the new approach on some classes of examples.
There are several natural extensions and further research directions that are left open by this work.
{We now sketch some of them.}

\paragraph*{How to deal with many cliques.} 
In the new ideal-sparse approach, instead of a single measure on the full space $\R^n$, one has several measures on smaller spaces indexed by the maximal cliques of the graph $G$. At any given level $t\ge 1$, the corresponding ideal-sparse bounds are at least as good as their dense analogs and, depending on the number of maximal cliques, their computation can be much faster.  The computation of the ideal-sparse parameters indeed involves several (based on the maximal cliques)  semidefinite matrices of smaller sizes. The first research direction is to investigate the trade-off between having  many cliques (in the ideal-sparse setting) and large matrix constraints (in the dense setting). As seen in Section \ref{secEDM} the sparse hierarchy behaves particularly bad on examples where the underlying graph has exponentially many cliques.
We suggest a possible solution in Remark \ref{remsparsechordal}, where we consider merging some of the cliques by considering a  (possibly chordal) extension of the support graph $G$. Clique merging has been explored before in the context of power flow networks, see \cite{20SAALT} and \cite{20GCG}. These methods exploit correlative sparsity and thus require the underlying support graph to be chordal. Finding the minimal chordal extension of a graph is NP-complete \cite{arnborg1987complexity}, 
but heuristics exist for certain cases (see, e.g., \cite{bodlaender2010treewidth}). 
Supposing one has chosen a method for finding chordal extensions, it is still unclear which among the possible chordal extensions will result in better SDPs. One can try to merge small cliques based on how much it would reduce the estimated computational burden. These estimates can be based, for example, on the number of constraints, see \cite{6510541}, or on the cost of an interior-point method iteration, see \cite{Sliwak2021ASO}.
As it stands, we know of no systematic way to find a ``computationally optimal" trade-off between the dense and ideal-sparse hierarchies.

\paragraph*{Application to other matrix factorization ranks.}
We have explored the application to  nonnegative  and  completely positive matrix factorization ranks. 
We have not considered their non-commutative analogs for the positive semidefinite (psd) rank and the completely positive semidefinite (cpsd) rank, where, respectively,  given  $M\in \R^{m\times n}_+$ one wants psd matrices $X_i,Y_j\in \mathcal S^r$ such that 
$M=(\langle X_i,Y_j\rangle)_{i\in [m], j\in [n]}$, and given  $A\in\mathcal S^n$ one wants psd matrices $X_i\in \mathcal S^r_+$ such that $A=(\langle X_i,X_j\rangle)_{i,j\in [n]}$, with $r$ smallest possible. One recovers the nonnegative rank and the cp-rank when restricting the factors $X_i,Y_j$ to be diagonal matrices. We refer the reader  to \cite{GdLL2019a}, where a common polynomial optimization framework is offered to treat all these four matrix factorization ranks. 
In the noncommutative setting of the psd- and cpsd-ranks, zero entries of $M$ (or $A$) also imply  ideal-type constraints of the form $X_iY_j=0$ (or $X_iX_j=0$). Thus the techniques in the present paper may extend to this general setting. We leave this extension to future work.  

\paragraph*{More general ideal-sparsity and applications.}
We have considered an ideal-sparsity structure, where the ideal in (\ref{eqIE}) is generated by quadratic monomials.
Beside their use for bounding matrix factorization ranks,   constraints of the form $x_ix_j = 0$ naturally arise in a number of other applications.  First we note that up to a change of variables, one can consider  more general constraints of the form  $(a^\top x + b)(c^\top x+d) = 0 $. This type of constraint is commonly referred to as a \emph{complementarity constraint}, where either the term $(a^\top x + b)$ or the term $(c^\top x+d)$ is required to be zero. We mention two areas where such complementary constraints naturally arise: analysis of neural networks and optimality conditions in optimization.

Complementarity constraints arise naturally when modeling neural networks with the rectified linear activation functions (ReLU). The semialgebraic representation of the graph of the ReLU function involves a constraint of the form $y(y-x) = 0$, which is exactly a complementarity constraint. The fact that the graph of the ReLU function admits a semialgebraic representation has been exploited computationally using the moment-sum-of-squares framework, for analyzing the Lipschitz constant of the neural network as well as stability and performance properties of dynamical systems controlled by the ReLU neural networks, see, e.g., ~\cite{chen2020semialgebraic,chen2021monotone,korda2022stability}. Ideal sparsity is therefore a natural candidate to render these methods more computationally efficient and would deserve further study.

Complementarity systems arise also in optimization within the Karush-Kuhn-Tucker (KKT) conditions. The complementarity slackness of the KKT condition reads $\lambda_i f_i(x) = 0$, where $\lambda_i$ is the Lagrange multiplier associated to the $i^{\mathrm{th}}$ constraint $f_i(x) \le 0$. If $f_i$ is affine, this is in the form of ideal constraints. The fact that the KKT conditions form a basic semialgebraic set when the optimization problem has polynomial data was exploited in~\cite{korda2017stability} to analyze dynamical systems controlled by optimization algorithms, albeit without exploiting the ideal-sparsity. More generally, the ideal-sparsity could be used to analyze the \emph{linear complementarity problems} (LCP) that have applications in, e.g., economics, engineering or game theory; see~\cite{cottle2009linear} for an extensive treatment of the subject.

Finally, instead of considering an ideal generated by quadratic monomials, 
one may consider an ideal generated by a set of monomials $x^S=\prod_{i\in S} x_i$ ($S\in \mathcal S$), where $\mathcal S$ is a given collection of subsets of $V=[n].$ The treatment extends naturally to this more general setting, where in the definition (\ref{eqsetK}) of the set $K$, we replace the constraints $x_ix_j=0$ ($\{i,j\}\in \olE$) by $\prod_{i\in S}x_i=0$ ($S\in \mathcal S$).
 Indeed,  let $V_1,\ldots,V_p$ denote  the maximal subsets of $V$ that do not contain any set $S\in \mathcal S$. Then, for the dense formulation (\ref{opt:dense}) of GMP, one can again show an equivalent sparse reformulation as in (\ref{opt:sparse}), which involves $p$ measures supported on the subspaces $\R^{|V_1|},\ldots,\R^{|V_p|}$ instead of a single measure on $\R^{|V|}$. We leave it for further research to explore  applications of this more general ideal-sparsity setting and possible further extensions to other types of varieties.

\subsection*{Acknowledgements}
We thank two anonymous referees for their comments and suggestions that helped improve the presentation.

\end{document}